\documentclass[10pt]{article}
\usepackage[top=2.5cm, bottom=2.5cm, left=2.5cm, right=2.5cm] {geometry}

\usepackage{amsmath,amssymb,theorem,color}
\numberwithin{equation}{section} % danh socong thuc theo section
\usepackage{amssymb}
\usepackage{amsmath}
\usepackage{graphicx}
\usepackage{amsfonts}%
\usepackage{amsmath,amssymb,theorem}
\newcommand{\R}{\ensuremath{\mathbb{R}}}
\newcommand{\N}{\ensuremath{\mathbb{N}}}

\newcommand{\ltn}{\ensuremath{\left| \! \left| \! \left|}}
\newcommand{\rtn}{\ensuremath{\right| \! \right| \! \right|}}

\newtheorem{theorem}{Theorem}[section]
{ \theorembodyfont{\normalfont} %\theorembodyfont{\rmfamily}

\newtheorem{remark}[theorem]{Remark}
}

\newtheorem{definition}[theorem]{Definition}
\newtheorem{lemma}[theorem]{Lemma}
\newtheorem{corollary}[theorem]{Corollary}
\newtheorem{proposition}[theorem]{Proposition}

\newcounter{enumctr}

\newcommand{\tn}[1]{{\left\vert\kern-0.25ex\left\vert\kern-0.25ex\left\vert #1 
		\right\vert\kern-0.25ex\right\vert\kern-0.25ex\right\vert}}

\begin{document}	

\title{Nonautonomous Young differential equations revisited}
\author{
Nguyen Dinh Cong\thanks{Institute of Mathematics, Vietnam Academy of Science and Technology, Vietnam {\it E-mail: ndcong@math.ac.vn}}, $\;$ Luu Hoang Duc\thanks{Institute of Mathematics, Vietnam Academy of Science and Technology, \& Max-Planck-Institut f\"ur Mathematik in den Naturwissenschaften, Leipzig, Germany {\it E-mail: lhduc@math.ac.vn, duc.luu@mis.mpg.de}}, $\;$ Phan Thanh Hong \thanks{Thang Long University, Hanoi, Vietnam {\it E-mail: hongpt@thanglong.edu.vn }}}
\date{}
\maketitle

\begin{abstract}
In this paper we prove that under weak conditions a nonautonomous Young differential equation possesses a unique solution which depends continuously on initial conditions. The proofs use estimates in $p$-variation norms,  greedy time techniques, and Gronwall-type lemma with the help of Shauder theorem of fixed points. 
\end{abstract}


%%%%%%%%%%%%%%%%%%%%%%%%%%%%%%%%%%%%%%%%%%%%

\section{Introduction}
This paper deals with the Young differential equation of the form
\begin{equation}\label{eqn.sde-fBm1}
dx_t = f(t,x_t)dt+g(t,x_t)d\omega_t, \ t\geq 0
\end{equation}
where $f: \R\times \R^d \to \R^d$ and $g: \R \times \R^d \to \R^{d \times m}$ are continuous functions, $\omega$ is a $\R^m$-valued function of  finite $p$-variation norm for some $1<p<2$. Such type of system is generated from  stochastic differential equations driven by fractional Brownian noises, as seen e.g. in \cite{mandelbrot}. Equation \eqref{eqn.sde-fBm1} is understood in the integral form
\begin{equation}\label{integral.eq.Stieltes0}
x_t=x_{0} +\int_0^t f(s,x_s) ds + \int_0^tg(s,x_s)d\omega_s,\  t\geq 0,
\end{equation}
where the first integral is of Riemannian type, meanwhile the second integral can be defined in the Young sense \cite{young}. The existence and uniqueness of the solution of \eqref{integral.eq.Stieltes0} are studied by several authors. When  $f,g$ are time-independent, system \eqref{integral.eq.Stieltes0} is proved in \cite{young} and \cite{ruzmaikina} and \cite{lyons94} to have a unique solution in a certain space of continuous functions with bounded  $p$-variation. The result is then generalized for the case $2<p<3$ by \cite{friz} and \cite{lyonsqian} using rough path theory, see also recent work by \cite{riedelscheutzow} for rough differential equations. According to their settings, $g$ is often assumed to be infinitely differentiable and bounded in itself and its derivatives.\\
Another approach following Z\"ahle \cite{zaehle} by using fractional derivatives can be seen in \cite{nualart3} which derives very weak conditions for $f$ and $g$ in \eqref{eqn.sde-fBm1}, in particular $g$ need to be only $C^1$ with bounded and H\"older continuous first derivative, to ensure the existence and uniqueness of the solution in the space of H\"older continuous functions. \\
%%%%%%%%%%%%%%%%%%%%%
Our aim in this paper is to close the gap between the two methods by proving that, under similar assumptions to those of Nualart and Rascanu~\cite{nualart3}, the existence and uniqueness theorem for system \eqref{eqn.sde-fBm1} still holds in the space of continuous functions with bounded $p$-variation norm. For that to work, we construct a sequence of the so-called greedy times (see e.g. \cite{lyonsetal}) such that the solution can be proved to exists uniquely in each interval of the consecutive greedy times, and is then concatenated to form a global solution. It is remarkable that since we are using estimates for $p$-variation norms, we do not apply the classical arguments of contraction mappings, but use Shauder-Tychonoff fixed point theorem as seen in \cite{lyons94} and a Gronwall-type lemma. \\
%Another issue is the generation of topological flow which was asserted in \cite{lyonsqian} for the autonomous system in any finite %time interval. It is however not clear how to construct a semigroup from autonomous systems when considering the system in the %whole real line time. We extend the result further for nonautonomous system \eqref{eqn.sde-fBm1} by not considering the system %in the extended space (which means adding the Riemann differential $dt$ into the Young differential $d\omega$), but by %constructing the shift dynamical system in the extended space of finite $p-$variation norm for the whole real line time and proving %that the system generates a continuous two-parameter flow. In particular, the autonomous system generates a continuous skew %product flow in the extended space in the sense of Sell \cite{sell}. The readers are recommended to \cite{GASch} and %\cite{riedelscheutzow} for similar studies of how to generates a random dynamical system from stochastic systems with fractional %Brownian motions.\\
The paper is organized as follows. In section 2, the Young integral is introduced and a greedy times analysis is given. In Section 3, we prove the existence and uniqueness of the global solution of system \eqref{integral.eq.Stieltes0} in Theorem \ref{theo2}, for this we need to formulate a Gronwall-type lemma. Proposition \ref{growth}  gives an estimate of $q$-var norm of solution via $p$-var norm of the driver $\omega$. We also prove the existence and uniqueness of the solution of the backward equation \eqref{RDEbw} in Theorem \ref{theo3}. In Section 4, the fact in Theorem \ref{flow1} that two trajectories do not intersect help to conclude that the Cauchy operator or the Ito map of \eqref{integral.eq.Stieltes0} generates a continuous two parameter flow.%, which in the autonomous case is a continuous nonautonomous dynamical system and helps to form a topological skew product flow in the sense of Sell \cite{sell}. 
%concludes that the solution is of exponential growth
%%%%%%%%%%%%%%%%%%%%%%%%%%%%%%%%%%%%%%%%%%%%

\section{Preliminaries}

\subsection{Young integral}\label{subsec.frac.int}
In this section we recall some facts about Young integral, more details can be seen in \cite{friz}. Let $C([a,b],\R^d)$ denote the space of all continuous paths $x:\;[a,b] \to \R^d$ equipped with sup norm $\|\cdot\|_{\infty,[a,b]}$ given by $\|x\|_{\infty,[a,b]}=\sup_{t\in [a,b]} |x_t|$, where $|\cdot|$ is the Euclidean norm in $\R^d$. For $p\geq 1$ and $[a,b] \subset \R$, a continuous path $x:[a,b] \to \R^d$ is of finite $p$-variation if 
\begin{eqnarray}
\ltn x\rtn_{p\text{-var},[a,b]} :=\left(\sup_{\Pi(a,b)}\sum_{i=1}^n|x_{t_{i+1}}-x_{t_i}|^p\right)^{1/p} < \infty,
\end{eqnarray}
where the supremum is taken over the whole class of finite partition of $[a,b]$. The subspace $\widehat{C}^p([a,b],\R^d)\subset C([a,b],\R^d)$ of all paths  $x$ with finite $p$-variation and equipped with the $p$-var norm
\begin{eqnarray*}
	\|x\|_{p\text{-var},[a,b]}&:=& |x_a|+\ltn x\rtn_{p\text{-var},[a,b]},
\end{eqnarray*}
is a nonseparable Banach space \cite[Theorem 5.25, p.\ 92]{friz}. Notice that if $x\in \widehat{C}^p([a,b],\R^d) $ then the mapping $(s,t)\to \ltn x\rtn_{p\text{-var},[s,t]}$ is continuous on the simplex $\Delta[a,b]:=\{(s,t)|a\leq s\leq t\leq b\}$, see \cite[Proposition 5.8, p. 80]{friz}. 

Furthermore, the closure of $C^\infty([a,b],\R^d)$ in $\widehat{C}^p([a,b],\R^d)$ is a separable Banach space denoted by  $\widehat{C}^{0,p}([a,b],\R^d)$ which can be defined as the space of all continuous paths $x$ such that
\[
\lim \limits_{\delta \to 0} \sup_{\Pi(a,b), |\Pi| \leq \delta} \sum_i |x_{t_{i+1}}-h_{t_i}|^p =0.
\]
It is easy to prove (see \cite[Corollary 5.33, p. 98]{friz}) that for $1\leq p< p'$ we have
\[
\widehat{C}^p([a,b],\R^d)\subset \widehat{C}^{0,p'}([a,b],\R^d).
\]

Also, for $0<\alpha \leq 1$ denote by $C^{\alpha\text{-Hol}}([a,b],\R^d)$ the Banach space of all H\"older continuous paths $x:[a,b]\to \R^d$ with exponential $\alpha$, equipped with the norm
\begin{eqnarray}
\|x\|_{\alpha\text{-Hol},[a,b]} &:=& |x_a| + \ltn x\rtn_{\alpha\text{-Hol},[a,b]}\notag\\
&=&  |x_a| +\sup_{(s,t)\in \Delta[a,b]} \frac{|x_t-x_s|}{(t-s)^{\alpha}}< \infty.
\end{eqnarray}
Clearly, if $x\in C^{\alpha\text{-Hol}}([a,b],\R^d)$ then for all $s,t\in [a,b]$ we have
\begin{eqnarray*}%\label{pvar-Hol.norm}
|x_t-x_s|&\leq& \ltn x\rtn_{\alpha\text{-Hol},[a,b]} |t-s|^\alpha.
\end{eqnarray*}
%Due to \cite[Proposition 5.10, p. 83]{friz},
Hence, for all $p$ such that $p\alpha\geq 1$ we have
\begin{eqnarray}\label{pvar-Hol.norm}
\ltn x\rtn_{p\text{-var},[a,b]}&\leq& \ltn x\rtn_{\alpha\text{-Hol},[a,b]} (b-a)^\alpha<\infty.
\end{eqnarray} 
Therefore, $C^{1/p\text{-Hol}}([a,b],\R^d)\subset \widehat{C}^p([a,b],\R^d)$.

As introduced in \cite{nualart3}, the space $W_b^{1/p,\infty}([a,b],\R^d)$ of measurable
functions $g :[a,b] \rightarrow \R^d$ such that 
$$\sup_{a<s<t<b}\left(\frac{|g_t-g_s|}{(t-s)^{1/p}}+\int_s^t\frac{|g_y-g_s|}{(y-s)^{1+1/p}}dy\right)<\infty$$
is a subspace of $C^{1/p\text{-Hol}}([a,b],\R^d)$. Hence $W_b^{1/p,\infty}([a,b],\R^d)\subset \widehat{C}^p([a,b],\R^d) $.
\begin{lemma}\label{additive}
Let $x\in \widehat{C}^{p}([a,b],\R^d)$, $p\geq 1$. If $a = a_1<a_2<\cdots < a_k = b$, then 
$$
\sum_{i=1}^{k-1}\ltn x\rtn^p_{p\text{-}\rm{var},[a_i,a_{i+1}]}\leq \ltn x\rtn^p_{p\text{-}\rm{var},[a_1,a_k]}\leq (k-1)^{p-1}\sum_{i=1}^{k-1}\ltn x\rtn^p_{p\text{-}\rm{var},[a_i,a_{i+1}]}.
$$ 
\end{lemma}
\begin{proof}
The proof is similar to the one in \cite[p.\ 84]{friz}, by using triangle inequality and power means inequality
 \begin{equation*}
 \frac{1}{n}\sum_{i=1}^n z_i \leq \left(\frac{1}{n}\sum_{i=1}^nz_i^r\right)^{1/r},\quad \forall z_i\geq 0, r\geq 1.
 \end{equation*}
\end{proof}
\begin{definition}
A continuous map $\overline{\omega}: \Delta[a,b]\longrightarrow \R^+$ is called a control if it is zero on the diagonal and superadditive, i.e\\
(i), For all $t\in [a,b]$, $\overline{\omega}_{t,t}=0$,\\
(ii), For all $s\leq t\leq u$ in $[a,b]$, $\overline{\omega}_{s,t}+\overline{\omega}_{t,u}\leq \overline{\omega}_{s,u}$.
\end{definition}
The functions $(s,t)\longrightarrow (t-s)^{\theta}$ with $\theta \geq 1$, and  $(s,t)\longrightarrow \ltn x\rtn^q_{p\text{-var},[s,t]}$, where $x$ is of bounded $p$-variation norm on $[a,b]$ and $q\geq p$ are some examples of control function. The following lemma gives a useful property of controls in relation with variations of a path (see \cite{friz} for more properties of control functions).
\begin{lemma}\label{controled}
Let $\overline{\omega}^j$ be a finite sequence of control functions on $[0,T]$, $C_j>0,\;\;j=\overline{1,k}$, $p\geq 1$ and  $x:[0,T]\rightarrow \R^d$ be a continuous path satisfying
$$|x_t-x_s|\leq \sum_{i=j}^k C_j \overline{\omega}^j(s,t)^{1/p},\;\; \forall s<t \in [0,T].$$
Then 
\begin{eqnarray}
\ltn x\rtn_{p\text{-}{\rm {var}},[s,t]}\leq \sum_{j=1}^k C_j \overline{\omega}^j(s,t)^{1/p},\;\; \forall s<t \in [0,T].
\end{eqnarray}
\end{lemma}
\begin{proof}
Consider an arbitrary finite partition $\Pi = (s_i)$, $i=0\ldots, n+1$, of $[s,t]$. By assumption and Minskowski inequality we have
\begin{eqnarray*}
\left(\sum_{i=0}^n|x_{s_{i+1}}-x_{s_i}|^p\right)^{1/p}&\leq & \left[\sum_{i=0}^n\left(\sum_{j=1}^k C_j \overline{\omega}^j(s_i,s_{i+1})^{1/p}\right)^p\right]^{1/p}\\
&\leq & \left(\sum_{i=0}^n C_1^p\overline{\omega}^1(s_i,s_{i+1})\right)^{1/p}+\dots+ \left(\sum_{i=0}^nC_k^p\overline{\omega}^k(s_i,s_{i+1})\right)^{1/p} \\
&\leq &  C_1 \left(\sum_{i=0}^n\overline{\omega}^1(s_i,s_{i+1})\right)^{1/p}+\dots+ C_k\left(\sum_{i=0}^n\overline{\omega}^k(s_i,s_{i+1})\right)^{1/p} \\
&\leq & \sum_{j=1}^k C_j  \overline{\omega}^j(s,t)^{1/p}.
\end{eqnarray*}
This implies the conclusion of the lemma.
\end{proof}
%===================================
%\medskip
%%%%%%%%%%%%%%%%%%%%%%%%%%%%%%

Now, consider $x\in \widehat{C}^{q}([a,b],\R^{d\times m})$ and $\omega\in \widehat{C}^p([a,b],\R^m)$, $p,q \geq 1$, if Riemann-Stieltjes sums for finite partition $\Pi=\{ a=t_0<t_1<\cdots < t_n=b \}$ of $[a,b]$ and any $\xi_i \in [t_i,t_{i+1}]$
\begin{eqnarray}\label{RSdef}
S_\Pi &:=& \sum_{i=1}^n x_{\xi_i}(\omega_{t_{i+1}}-\omega_{t_i}) ,
\end{eqnarray}
converges as the mesh $|\Pi| := \displaystyle\min_{0\leq i \leq n-1} |t_{i+1}-t_i|$  tends to zero, we call the limit is the Young integral of $x$ w.r.t $\omega$ on $[a,b]$ denoted by $\int_a^b x_td\omega_t$.
It is well known that if $p,q\geq 1$ and $\frac{1}{p}+\frac{1}{q}  > 1$, the Young integral $\int_a^bx_td\omega_t$ exists (see \cite[p.\ 264--265]{young}). 
Moreover, if $x^n$ and $\omega^n$ are of bounded variation, uniformly bounded in  $\widehat{C}^q([a,b],\R^{d\times m})$,  $\widehat{C}^p([a,b],\R^{d\times m})$ and  converges uniformly to $x$, $\omega$ respectively, then the sequence of the Riemann-Stieljes integral $\int_a^b x^n_td\omega^n_t$ approach $\int_a^bx_td\omega_t$ as $n\to \infty$ (see \cite{friz}). This integral satisfies additive property by the construction, and the so-called Young-Loeve estimate \cite[Theorem 6.8, p.\ 116]{friz}
\begin{equation}\label{YL0}
\Big|\int_s^t x_ud\omega_u-x_s[\omega_t-\omega_s]\Big| \leq K \ltn x\rtn_{q\text{-var},[s,t]} \ltn\omega\rtn_{p\text{-var},[s,t]},
\end{equation}
where 
\begin{equation}\label{constK}
K:=(1-2^{1-\theta})^{-1},\qquad \theta := \frac{1}{p} + \frac{1}{q} >1.
\end{equation}

\begin{lemma}\label{lemma1}
	For $1\leq p, 1\leq q$ such that $\theta = \frac{1}{p}+\frac{1}{q}>1$ and $x\in \widehat{C}^{q}([a,b],\R^{d\times m})$, $\omega\in \widehat{C}^p([a,b],\R^m)$, the following estimates hold
	\begin{eqnarray}
	\ltn\int_a^.x_ud\omega_u\rtn_{p\text{-}{\rm {var}},[a,b]}&\leq& \ltn\omega\rtn_{p\text{-}{\rm {var}},[a,b]}\left(|x_a|+(K+1)\ltn x\rtn_{q\text{-var},[a,b]}\right),\label{YL1}
	\end{eqnarray}
	where $K$ is determined by \eqref{constK}.
\end{lemma}
\begin{proof}
	To prove \eqref{YL1}, we note that by virtue of \eqref{YL0} for all $a\leq s\leq t\leq b$ we have
		\begin{eqnarray}
		\left|\int_s^tx_ud\omega_u\right|&\leq & |x_s||(\omega_t-\omega_s)|+ K\ltn\omega\rtn_{p\text{-var},[s,t]}\ltn x\rtn_{q\text{-var},[s,t]}\notag\\
		&\leq &  \ltn\omega\rtn_{p\text{-var},[s,t]}\left(\|x\|_{\infty,[a,b]} + K \ltn x\rtn_{q\text{-var},[a,b]}\right)\notag\\
		&\leq &  \ltn\omega\rtn_{p\text{-var},[s,t]}\left(|x_a|+ (K+1) \ltn x\rtn_{q\text{-var},[a,b]}\right).\notag
		\end{eqnarray}
		
		Since $\bar{\omega}(s,t):=\ltn\omega\rtn^p_{p\text{-var},[s,t]}$ is a control on the simplex $\Delta[a,b]$ (see \cite[Proposition 5.8, p. 80]{friz}), $\int_a^.x_ud\omega_u$ is of bounded $p$-variation and  
		$$\ltn\int_a^.x_ud\omega_u\rtn_{p\text{-var},[a,b]}\leq \ltn\omega\rtn_{p\text{-var},[a,b]}\left(|x_a|+(K+1)\ltn x\rtn_{q\text{-var},[a,b]}\right)$$
		due to  \cite[Proposition 5.10(i), p. 83]{friz}.% in Friz and Victoir.
	%\end{enumerate}
\end{proof}
%==============================
\medskip

Due to Lemma \ref{lemma1}, the integral $t\mapsto \int_a^tx_sd\omega_s$ is a continuous bounded $p$-variation path.
Note that the definition of Young integral does depend on the direction of integration in a simple way like the Riemann-Stieltjes integral. Namely, it is easy to see that
\begin{eqnarray}
\int_b^ax_ud\omega_u = \displaystyle\lim_{\Pi(a,b), |\Pi|\to 0} \sum_{i=1}^n x_{\xi_i}(\omega_{t_{i}}-\omega_{t_{i+1}}) = - \displaystyle\lim_{\Pi(a,b), |\Pi|\to 0} \sum_{i=1}^n x_{\xi_i}(\omega_{t_{i+1}}-\omega_{t_i}) = -\int_a^bx_ud\omega_u.\label{Y-integral-direction}
\end{eqnarray}

\subsection{Greedy times analysis}\label{subsec.greedy-times}
%First, for any $-\infty < a\leq 0\leq b < \infty $ we introduce the notation
%$$
%\widehat{C}^{0,p}_0([a,b],\R^m):=\{\omega\in  \widehat{C}^{0,p}([a,b],\R^m)|\;\omega_0 = 0\}.
%$$
Denote by $\widetilde{C}^{p}(\R,\R^m)$ the space of all continuous functions $\omega: \R \to \R^m$ such that for any $T>0$ the restrictions of $\omega$ to $[-T,T]$ is of $\widehat{C}^{p}([-T,T],\R^m)$. Equip $\widetilde{C}^{p}(\R,\R^m)$ with the metric
\[
d(\omega^1,\omega^2) := \sum_{n=1}^\infty 2^{-n} \frac{\|\omega^1 - \omega^2\|_{p\text{-var},[-n,n]}}{1+\|\omega^1 - \omega^2\|_{p\text{-var},[-n,n]}}.
\]
Let $n\in \N$, observe that metric $d$ satisfies
 \begin{equation}\label{metricd}
 \begin{array}{rcl}
d(\omega^1,\omega^2) &\leq& \|\omega_1 - \omega^2\|_{p\text{-var},[-n,n]} + 2^{-n},\\[6pt] 
 \|\omega^1 -\omega^2\|_{p\text{-var},[-n,n]} &\leq& \frac{2^n d(\omega^1,\omega^2)}{1- 2^n d(\omega^1,\omega^2)},
 \end{array}
\end{equation}
where the second inequality holds for any fixed $n$ and $\omega^1, \omega^2$ close enough such that $2^n d(\omega^1,\omega^2) <1$. Hence every Cauchy sequence $(\omega^k)_k$ w.r.t. metric $d$ is also a Cauchy sequence when restricted to $\widehat{C}^p([-n,n],\R^m)$, thus converges to a limit $\omega^*\in \widehat{C}^p([-n,n],\R^m)$ which is uniquely defined pointwise, so $\omega^* \in \widetilde{C}^{p}(\R,\R^m)$. Therefore, $(\widetilde{C}^{p}(\R,\R^m),d)$ is a complete metric space.%, especially, $(\widetilde{C}^{p}(\R,\R^m),d)$ is a Polish space.
\begin{remark}\label{truncated}
(i) {\em Truncation:} Another consequence of \eqref{metricd} is that the truncated version of $\omega \in \widetilde{C}^{p}(\R,\R^m)$ in any $\widehat{C}^p([-n,n],\R^m)$ differs very small w.r.t. metric $d$ from the original $\omega$ if we choose $n$ large enough. Moreover, if a function is continuous w.r.t. $\omega$ on any restriction in $\widehat{C}^p([-n,n],\R^m)$ for any $n >0$ then it is also continuous w.r.t. $\omega$ in $\widetilde{C}^{p}(\R,\R^m)$ with respect to metric $d$. 

(ii) {\em Concatenation:} Let $a<b<c$. Suppose that
$ \omega^1 \in \widehat{C}^{p}([a,b],\R^m)$, $\omega^2 \in \widehat{C}^{p}([b,c],\R^m)$ and 
$\omega^1_b=\omega^2_b$. Then $\omega^1. 1_{[a,b]} +\omega^2.1_{[b,c]}$ belongs to $\widehat{C}^{p}([a,c],\R^m)$. 
\end{remark}
%\medskip

For  any given $\lambda,\mu>0$ we construct a strict increasing sequence of greedy times $\{\tau_n\}$, 
$$
\tau_n:\widetilde{C}^{p}(\R,\R^m)\longrightarrow\R^+,
$$  
such that $\tau_0\equiv 0$ and 
\begin{equation}\label{stoppingtime}
|\tau_{i+1}(\omega)-\tau_i(\omega)|^\lambda+\ltn\omega\rtn_{p\text{-var},[\tau_i(\omega),\tau_{i+1}(\omega)]} = \mu.
\end{equation}
To do so, first define $\tau: \widetilde{C}^{p}(\R,\R^m)\longrightarrow\R^+$ such that
\[
\tau (\omega) := \sup \{t\geq 0: t^\lambda + \ltn \omega \rtn_{p\text{-var},[0,t]} \leq \mu\}.
\]  
Observe that the function $\kappa(t):= t^\lambda + \ltn \omega \rtn_{p\text{-var},[0,t]}$ is continuous and stricly increasing w.r.t. $t$ with $\kappa (0) =0$ and $\kappa(\infty) = \infty$, therefore due to the continuity there exists a unique $\tau= \tau(\omega)>0$ such that
\begin{equation}\label{stoptime1}
\tau^\lambda + \ltn \omega \rtn_{p\text{-var},[0,\tau]} = \mu.
\end{equation}
%Thus $\tau$ is well defined and our next lemma shows its continuity.
Thus $\tau$ is well defined. Next, we construct the so-called greedy times inductively as follows. 
Set $\tau_0 := 0$, $\tau_1 (\omega) := \tau(\omega) $. Suppose that we have defined $\tau_n (\omega)$ for $n\geq 1$, looking at the following equality as an equation of $\delta_n(\omega)\in \R^+$,  like above we find an unique $\delta_n(\omega)$
 such that
$$
\mu = \delta_n^\lambda(\omega) + \ltn \omega(\cdot + \tau_n(\omega))\rtn_{p-var,[0,\delta_n(\omega)]},
$$
hence we can set
\begin{equation}\label{stoptime2}
\tau_{n+1}(\omega) := \tau_{n-1}(\omega) + \delta_n(\omega), 
\end{equation}
where $\delta_n(\omega)$ is determined above. Thus we have defined a sequence of greedy times $\{\tau_n\}$ for all $n=0,1,2,\ldots$.
Such a sequence of greedy times then satisfies \eqref{stoppingtime}.% and $\tau_n(\omega) \leq n \mu^{\frac{1}{\lambda}}$. Moreover, we prove that
%=================================
\medskip

Now, we fix $\omega\in \widetilde{C}^{p}(\R,\R^m)$ and consider the number of greedy times inside an arbitrary finite interval of $\R^+$. We write $\tau_n$ for $\tau_n(\omega)$ to simplify the notation. For given $T>0$, we introduce the notation 
\begin{equation}\label{N}
N(T,\omega):= \sup \{n: \tau_n\leq T\}<\infty.
\end{equation}
or more general, for any $0\leq a < b <\infty$,
\begin{eqnarray}
N(a,b,\omega)&:=& \sup\{n: \tau_n\leq b\}- \inf\{n: \tau_n\geq a\}.
\end{eqnarray}
%====================
\begin{lemma}
Let $p'\geq \max\{p,\frac{1}{\lambda}\}$ be arbitrary, the following estimate holds
\begin{eqnarray}\label{mu2}
N(T,\omega)&\leq& \frac{2^{p'-1}}{\mu^{p'}} \Big( T^{p'\lambda}+\ltn\omega\rtn^{p'}_{p\text{-var},[0,T]}\Big).
\end{eqnarray}
More general,
\begin{eqnarray}\label{Nab2}
N(a,b,\omega)&\leq& \frac{2^{p'-1}}{\mu^{p'} }\Big[ (b-a)^{p'\lambda}+\ltn\omega\rtn^{p'}_{p\text{-}\rm{var},[a,b]}\Big].
\end{eqnarray}
\end{lemma}
\begin{proof}
  We have for all $n\in \N^*$
\begin{eqnarray}\label{estN}
n\mu^{p'} &= & \sum_{i=0}^{n-1} \mu^{p'} = \sum_{i=0}^{n-1}\left[ |\tau_{i+1}-\tau_i|^{\lambda}+\ltn\omega\rtn_{p\text{-var},[\tau_i,\tau_{i+1}]}\right]^{p'} \notag\\
&\leq& 2^{p'-1}\left[\sum_{i=0}^{n-1} |\tau_{i+1}-\tau_i|^{p'\lambda }+ \Big(\sum_{i=0}^{n-1} \ltn\omega\rtn^p_{p\text{-var},[\tau_i,\tau_{i+1}]}\Big)^{p'/p}\right]\notag\\
&\leq& 2^{p'-1}\left[(\tau_n-\tau_0)^{p'\lambda }+ \Big(\sum_{i=0}^{n-1} \ltn\omega\rtn^p_{p\text{-var},[\tau_i,\tau_{i+1}]}\Big)^{p'/p}\right]\notag\\
&\leq &  2^{p'-1}\left[\tau_n^{p'\lambda} + \ltn\omega\rtn^{p'}_{p\text{-var},[0,\tau_n]}\right].
\end{eqnarray}
Consequently, we obtain
\begin{eqnarray*}%\label{mu2}
N(T,\omega)&\leq& \frac{2^{p'-1}}{\mu^{p'}} \Big[ T^{p'\lambda}+\ltn\omega\rtn^{p'}_{p\text{-var},[0,T]}\Big].
\end{eqnarray*}
Similarly, \eqref{Nab2} holds.
\end{proof}
%=========================================
\begin{remark}
\begin{enumerate}
\item  Since the left-hand side of \eqref{estN} tends to infinite its right hand side cannot be bounded. This implies that $\tau_n\to\infty$ as $n\to \infty$.
\item We can construct the sequence of greedy times starts at $\tau_0=t_0$, an arbitrary point in $\R$, and on $(-\infty, t_0])$ in a similar manner.
\item The original idea of greedy times was introduced in \cite{cassetal} for autonomous systems. A version of stopping times was developped before by \cite{GAMSch} and then by \cite{DGANS}. Here we propose another version of greedy times which fits with the nonautonomous setting.
\end{enumerate}
\end{remark}

%\end{document}
\section{Existence and uniqueness theorem}\label{subsec.FDE}
In this section, we are working with the restriction of any trajectory $\omega$ in a given time interval $[0,T]$ by consider it as an element in  $\widehat{C}^{p}([0,T],\R^m)$, for a certain $p \in (1,2)$ (see Remark \ref{truncated} for the relation between $\omega \in \widetilde{C}^{p}(\R,\R^m)$ and its restrictions). Consider the Young differential equation in the integral form as:
\begin{equation}\label{integral.eq.Stieltes}
x_t=x_{0} +\int_0^t f(s,x_s) ds + \int_0^tg(s,x_s)d\omega_s,\;\; t\in [0,T].
\end{equation}
We recall here a result in \cite{nualart3} on existence and uniqueness of solution of \eqref{integral.eq.Stieltes}, which was proved using contraction mapping arguments with $\omega$ in a Besov-type space. In this paper we however would like to derive a proof in $\widehat{C}^p$ applying Shauder fixed point theorem and greedy time tool. First we need to formulate some assumptions on  the coefficient functions $f$ and $g$ of \eqref{integral.eq.Stieltes}.\\
 
 (${\textbf H}_1$) $g(t,x)$ is differentiable in $x$ and there exist some constants $0<\beta,\delta\leq 1$, a control function $h(s,t)$ defined on $\Delta[0,T]$ and for every $N\geq 0$ there exists $M_N>0$ such that the following properties hold:
 $$
 (H_g) : \begin{cases}
 (i)\quad \hbox{Lipschitz continuity}\\
 |g(t,x) - g(t,y)| \leq L_g |x-y|,\quad \forall x,y \in\R^d, \quad \forall t\in [0,T],\\
 (ii) \quad \hbox{Local H\"older continuity}\\
 |\partial_{x_i}g(t,x)- \partial_{y_i}g(t,y)| \leq M_N |x-y|^\delta, \\
 \hspace*{3cm} \quad \forall  x,y\in\R^d,\; |x|,|y|\leq N,\quad \forall t\in [0,T],\\
 (iii) \quad \hbox{Generalized H\"older continuity in time} \\
 |g(t,x)-g(s,x)| + |\partial_{x_i}g(t,x) - \partial_{x_i}g(s,x)| \leq h(s,t)^\beta\\
  \hspace*{3cm} \quad  \forall x\in\R^d, \quad\forall s,t\in [0,T], s<t.
  \end{cases}
  $$
  
  (${\textbf H}_2$) There exists $a>0$ and $b\in L_{\frac{1}{1-\alpha}}([0,T],\R^d)$, where $\frac{1}{2} \leq \alpha <1 $, and for every $N\geq 0$ there exists $L_N >0$ such that the following properties hold:
  $$
 (H_f) : \begin{cases}
 (i)\quad \hbox{Local Lipschitz continuity}\\
 |f(t,x) - f(t,y)| \leq L_N |x-y|, \quad \forall x,y\in\R^d,\;|x|,|y|\leq N, \quad\forall t\in [0,T],\\
 (ii) \quad \hbox{Boundedness}\\
 |f(t,x)| \leq a |x| + b(t), \quad\forall x\in \R^d,\quad \forall t\in [0,T].\\
  \end{cases}
  $$
  
  (${\textbf H}_3$) The parameters in ${\textbf H}_1$ and ${\textbf H}_2$ statisfy the inequalities $\delta>p-1,\;\;\beta >1-\dfrac{1}{p},\;\; \delta\alpha > 1-\dfrac{1}{p}$. \\  

We would like to study the existence and uniqueness of the solution of \eqref{integral.eq.Stieltes} under the given conditions that $x\in \widehat{C}^{q}([0,T],\R^d)$ with appropriate constant $q>0$.

%=============================================================
%Put
%\begin{equation}
%\alpha:= \dfrac{\rho-1}{\rho} \label{alpha}
%\end{equation}
By the assumption $ p\in (1,2)$ and the condition ${\textbf H}_3$,  $1-\dfrac{1}{p}<\min\left\{\beta,\delta\alpha, \dfrac{\delta}{p},\dfrac{1}{2} \right\}$, thus we can choose consecutively constants $q_0,q$ such that
\begin{eqnarray}
1-\dfrac{1}{p}< &\dfrac{1}{q_0}& <\min\left\{\beta,\delta\alpha, \dfrac{\delta}{p},\dfrac{1}{2} \right\}, \label{q0}\\
\frac{1}{q_0\delta } \leq &\dfrac{1}{q} & < \min\{\alpha, \dfrac{1}{p}\}. \label{q} 
%\max\Big\{\frac{1}{q},1- \frac{1}{q_0} \Big\}<& \frac{1}{p}&< \nu. \label{p}
\end{eqnarray}
Then, we have
	\begin{equation}\label{relation-p-q-alpha}
	\frac{1}{p}+\frac{1}{q_0} > 1,\;\; q_0\beta >1,\;\; q_0\geq q_0\delta \geq q> p,\;\; q\alpha > 1.
	\end{equation}

%========================================================
We now consider $x\in \widehat{C}^{q}([t_0,t_1],\R^d)$ with some $[t_0,t_1] \subset[0,T]$. Define the mapping given by
%\[
%F: \widehat{C}^{q}([t_0,t_1],\R^d)\longrightarrow  \widehat{C}^{q}([t_0,t_1],\R^d)
%\]
%\textcolor{red}{here we haven't we prove the map valued in $\widehat{C}^p$ yet, after at proposition 4!}
\begin{eqnarray}
F(x)_t &= &x_{t_0} + I(x)_t+J(x)_t\nonumber\\
&:=& x_{t_0} + \int_{t_0}^t f(s,x_s) ds +\int_{t_0}^t g(s,x_s)d\omega_s, \quad\forall t\in [t_0,t_1].\label{eqn.F}
\end{eqnarray}
Note that a fixed point of $F$ is a solution of \eqref{integral.eq.Stieltes} on $[t_0,t_1]$ with the boundary condition $x(t_0)=x_{t_0}$ (the initial condition $x_0$ of \eqref{integral.eq.Stieltes} is then not given).
%in which $t_0,t_1\in [0,T]$ will be defined later.
%==============

%Denote 
Introduce the notations
\begin{eqnarray}\label{M}
M&:=&\max\{L_g,aT^{1-\alpha},|g(0,0)|+h(0,T)^\beta,\|b\|_{L_{\frac{1}{1-\alpha}}}\},\\ 
M'_N&:=&\max\{L_N,M_N,M\},\qquad \forall N>0.
\end{eqnarray}
%====================================
It can be seen from the above assumptions that $|g(t,x)|\leq |g(t,0)|+ L_g|x|$ and $|g(t,0)|\leq |g(0,0)|+ h(0,T)^{\beta} $, hence 
\begin{equation}\label{g}
|g(t,x)|\leq |g(0,0)|+ h(0,T)^{\beta}+ L_g|x| \leq M(1+|x|).
\end{equation}
For the next propositions we need the following auxiliary lemma.
%=====================================
\begin{lemma}\label{lemma2}
	Assume that ${\textbf H}_1-{\textbf H}_3$ are satisfied.\\
	(i) If $x \in \widehat{C}^{q}([t_0,t_1],\R^d)$ then $g(\cdot,x_.) \in \widehat{C}^{q_0}([t_0,t_1],\R^{d\times m})$ and 
	\begin{eqnarray}\label{gx}
	\ltn g(\cdot,x_.)\rtn_{q_0\text{-}\rm{var},[t_0,t_1]}\leq M(1+\ltn x\rtn_{q\text{-}\rm{var},[t_0,t_1]}).
	\end{eqnarray}
	(ii) For all $s< t$ and for all $x_i\in \R^d$  such that $|x_i|\leq N$, $i=1,2,3,4$, then
	\begin{eqnarray*}
	|g(s,x_1)-g(s,x_3)-g(t,x_2)+g(t,x_4)|&\leq& L_g|x_1-x_2-x_3+x_4|+ |x_2-x_4|h(s,t)^\beta\\
	&&+ M_N|x_2-x_4|(|x_1-x_2|^{\delta}+|x_3-x_4|^{\delta}).
	\end{eqnarray*}
	(iii) For any $x,y \in \widehat{C}^{q}([t_0,t_1],\R^d)$ such that $x_{t_0} = y_{t_0}$ and $\|x\|_{\infty,[t_0,t_1]}\leq N$, $\|y\|_{\infty,[t_0,t_1]}\leq N$ we have
	\begin{eqnarray}\label{gxy}
	\ltn g(\cdot,x_.)-g(\cdot,y_.)\rtn_{q_0\text{-}\rm{var},[t_0,t_{1}]}&\leq &M'_N\ltn x-y\rtn_{q\text{-}\rm{var},[t_0,t_{1}]} \left(2+\ltn x\rtn^{\delta}_{q\text{-}\rm{var},[t_0,t_{1}]}+\ltn y\rtn^{\delta}_{q\text{-}\rm{var},[t_0,t_{1}]}\right).\notag\\
		\end{eqnarray}
\end{lemma}

\begin{proof} (i)
For $s<t$ in $[t_0,t_1]$, we have
	\begin{eqnarray*}
		|g(t,x_t)-g(s,x_s)|&\leq & |g(t,x_t)-g(t,x_s)|+|g(t,x_s)-g(s,x_s)|\\
		&\leq & L_g|x_t-x_s| +  h(s,t)^{\beta}.
		%&\leq & L\|x\|_{p_0\text{-var},[s,t]}+\|h\|_{p_0\text{-var},[s,t]}
	\end{eqnarray*}
	Let $\Pi=(s_i)_1^{n+1}$ be an arbitrary finite partition of $[t_0,t_1]$, $s_1=t_0, s_{n+1}=t_1$. Since $q_0\geq q$ and $q_0\beta>1$ we have 
	\begin{eqnarray*}
		\left(\sum_{i=1}^{n}|g(s_{i+1},x_{s_{i+1}})-g({s_i},x_{s_i})|^{q_0}\right)^{1/q_0}&\leq & L_g \left(\sum_{i=1}^{n} |x_{s_{i+1}}-x_{s_i}|^{q_0}\right)^{1/q_0} +  \left(\sum_{i=1}^{n} h(s_{i},s_{i+1})^{q_0\beta}	\right)^{1/q_0}\\
		&\leq & L_g\ltn x\rtn_{q_0\text{-var},[t_0,t_1]}+ h(t_0,t_1)^{\beta}\\
		&\leq & L_g\ltn x\rtn_{q\text{-var},[t_0,t_1]}+ h(t_0,t_1)^{\beta} \\
		&\leq & L_g\ltn x\rtn_{q\text{-var},[t_0,t_1]}+ h(0,T)^{\beta} \\
		&\leq & M(1+\ltn x\rtn_{q\text{-var},[t_0,t_1]})<\infty.
	\end{eqnarray*}
	Take the superemum over the set of all finite partition $\Pi$ we get $g(\cdot,x_.) \in \widehat{C}^{q_0}([t_0,t_1],\R^{d\times m})$ and
	$$
	\ltn g(\cdot,x_.)\rtn_{q_0\text{-var},[t_0,t_1]}\leq M(1+\ltn x\rtn_{q\text{-var},[t_0,t_1]}).
	$$

	(ii) This part is similar to \cite[Lemma 7.1]{nualart3} with our function $h(s,t)^\beta$ playing the role of  $|t-s|^\beta$ in \cite[Lemma 7.1]{nualart3}.\\
	
	(iii) Note that $ q_0\beta>1$ and $ q_0\delta \geq q$ hence
	\begin{eqnarray}
	\ltn g(\cdot,x_.)-g(\cdot,y_.)\rtn_{q_0\text{-var},[t_0,t_{1}]}&:=& \left(\sup_{\Pi([t_0,t_1])} \sum_{i} | g(s_{i+1},x_{s_{i+1}})- g(s_{i+1},y_{s_{i+1}}) -g(s_i,x_{s_{i}}) +g(s_{i},y_{s_{i}}) |^{q_0}\right)^{1/q_0}\notag\\
	&\leq & L_g\sup_{\Pi([t_0,t_1])}\left(\sum_i |x_{s_{i+1}}-y_{s_{i+1}}-x_{s_i}+y_{s_i} |^{q_0}\right)^{1/q_0}+\notag\\
	&&\hspace*{-2cm} +  \|x-y\|_{\infty,[t_0,t_1]}\sup_{\Pi([t_0,t_1])}\left(\sum_i h(s_{i},s_{i+1})^{q_0\beta}\right)^{1/q_0}\notag\\
	&&\hspace*{-2cm}  + M_N \|x-y\|_{\infty,[t_0,t_1]} \sup_{\Pi([t_0,t_1])}\left[ \left(\sum_i |x_{s_{i+1}}-x_{s_i}|^{q_0\delta }\right)^{1/q_0}+\left(\sum_i |y_{s_{i+1}}-y_{s_i}|^{q_0\delta  }\right)^{1/q_0}\right]\notag\\
	&\leq& L_g\ltn x-y\rtn_{q\text{-var},[t_0,t_1]}\notag\\
	&&+ \|x-y\|_{\infty,[t_0,t_1]}\left[   h(t_0,t_1)^{\beta} + M_N \left(\ltn x\rtn^{\delta}_{q\text{-var},[t_0,t_{1}]}+\ltn y\rtn^{\delta}_{q\text{-var},[t_0,t_{1}]}\right)\right]\notag\\
	&\leq & M'_N\ltn x-y\rtn_{q\text{-var},[t_0,t_{1}]} \left(2+\ltn x\rtn^{\delta}_{q\text{-var},[t_0,t_{1}]}+\ltn y\rtn^{\delta}_{q\text{-var},[t_0,t_{1}]}\right).\notag
	\end{eqnarray}
	The lemma is proved.
\end{proof}
%=========================================

For a proof of our main theorem on existence and uniqueness of solutions of an Young  differential equation, we need the following proposition.
% conclusions, the proofs of which can be found in the Appendix.

\begin{proposition}\label{EstIJ}
	Assume that ${\textbf H}_1-{\textbf H}_3$ are satisfied.  Let $0\leq t_0<t_1\leq T$ be arbitrary, $q$ be chosen as above satisfying \eqref{q} and $F$ be defined by \eqref{eqn.F}. Then for any 
	$x\in \widehat{C}^{q}([t_0,t_1],\R^d)$ we have $F(x)\in \widehat{C}^{q}([t_0,t_1],\R^d)$, thus
	$$
	F: \widehat{C}^{q}([t_0,t_1],\R^d)\longrightarrow  \widehat{C}^{q}([t_0,t_1],\R^d).
	$$
	%where $F$ are defined by \eqref{eqn.F} and $q$ is chosen above satisfying \eqref{q}. 
	Moreover, the following statements hold
	
	(i) The $q$-variation of $F(x)$ satisfies
	\begin{eqnarray}
	\ltn F(x)\rtn_{q\text{-}\rm{var},[t_0,t_1]}&\leq& M(K+2)\left(1+\|x\|_{q\text{-}\rm{var},[t_0,t_1]}\right) \left((t_1-t_0)^{\alpha}+\ltn\omega\rtn_{p\text{-}\rm{var},[t_0,t_1]}\right).\label{Fx}
\end{eqnarray}	

(ii) Let $N\geq 0$ be arbitrary but fixed. Suppose that $x,y\in\widehat{C}^{q}([t_{0},t_{1}],\R^d)$ be such that $\|x\|_{\infty,[t_0,t_1]}\leq N$, $\|y\|_{\infty,[t_0,t_1]}\leq N$ and $x_{t_0} = y_{t_0}$, then we have
	\begin{eqnarray}
	\|F(x)-F(y)\|_{q\text{-}\rm{var},[t_0,t_{1}]}&\leq& \|x-y\|_{q\text{-}\rm{var},[t_0,t_{1}]}\left((t_1-t_0)+\ltn\omega\rtn_{p\text{-}\rm{var},[t_0,t_1]}\right)\notag\\
	&& \times M'_N (K+1) \left(2+\ltn x\rtn^{\delta}_{q\text{-}\rm{var},[t_0,t_{1}]}+\ltn y\rtn^{\delta}_{q\text{-}\rm{var},[t_0,t_{1}]}\right).\label{Fxy}
	\end{eqnarray}
\end{proposition}%
%==============
\begin{proof}%[Proof of Proposition \ref{EstIJ}]
(i)	Since $\frac{1}{p}+\frac{1}{q_0} > 1$, by virtue of \eqref{gx}, the Young integral $\int_0^tg(s,x_s)d\omega_s$ exists for all $t\in [t_0,t_1]$. Using \eqref{YL1}, \eqref{eqn.F} and \eqref{g} we get 
	\begin{eqnarray*}
		\ltn J(x)\rtn_{q\text{-var},[t_0,t_1]}&\leq&\ltn J(x)\rtn_{p\text{-var},[t_0,t_1]}\qquad \text {(because $p \leq q$)} \notag\\
		&\leq& \ltn\omega\rtn_{p\text{-var},[t_0,t_1]}\left[|g(t_0,x_{t_0})|+ (K+1) \ltn g(.,x_.)\rtn_{q_0\text{-var},[t_0,t_1]}\right]\notag\\
		&\leq & \ltn\omega\rtn_{p\text{-var},[t_0,t_1]}\left[M(1+|x_{t_0}|)+M(K+1)(1+\ltn x\rtn_{q\text{-var},[t_0,t_1]})\right]\\
		%&\leq & \|\omega\|_{q\text{-var},[s,t]}\left(C_0 +\|h\|_{\infty,p_0\text{-var},[s,t]}+ L\|x\|_{\infty,q_0\text{-var},[s,t]}  \right)\\
		&\leq &\ltn\omega\rtn_{p\text{-var},[t_0,t_1]}M\left[ (K+2) + |x_{t_0}| +(K+1)\ltn x\rtn_{q\text{-var},[t_0,t_1]}  \right].
	\end{eqnarray*}

Now, by H\"older inequality and the assumption ${\textbf H}_2$ we have 
	$$
	\int_s^t |b(u)|du\leq \left (\int_s^t|b(u)|^\frac{1}{1-\alpha} du \right)^{1-\alpha}  \left (\int_s^t 1 du \right)^\alpha\leq \|b\|_{L_{\frac{1}{1-\alpha}}} (t-s)^{\alpha}\leq M(t-s)^{\alpha}.
	$$ 
	Therefore, for $s<t$ in $[t_0,t_1]$ using the assumption ${\textbf H}_2$ we have %from \eqref{eqn.F} that
	\begin{eqnarray*}
	\left|\int_s^tf(u,x_u)du\right| &\leq & \int_s^t( a|x_u|+|b(u)|)du\\
	&\leq &  a\|x\|_{\infty,[s,t]}(t-s) + \|b\|_{L_{\frac{1}{1-\alpha}}} (t-s)^{\alpha}\\
	&\leq & (t-s)^{\alpha} \left(aT^{1-\alpha}\|x\|_{\infty,[t_0,t_1]} + \|b\|_{L_{\frac{1}{1-\alpha}}}\right)\\
	&\leq & (t-s)^{\alpha}M \left(1+|x_{t_0}|+\ltn x\rtn_{q\text{-var},[t_0,t_1]}\right).
	\end{eqnarray*}
This implies 
	\begin{eqnarray*}
	\ltn I(x)\rtn_{q\text{-var},[t_0,t_1]}=\ltn\int_{t_0}^.f(u,x_u)du\rtn_{q\text{-var},[t_0,t_1]}\leq M(t_1-t_0)^{\alpha} \left(1+|x_{t_0}|+\ltn x\rtn_{q\text{-var},[t_0,t_1]}\right)
	\end{eqnarray*}
by \cite[Proposition 5.10(i), p. 83] {friz} and the fact that the function $(s,t)\to (t-s)^{q\alpha}$ defined on $\Delta[t_0,t_1]$ is a control function for $q\alpha > 1$.
Since 
$$
\ltn F(x)\rtn_{q\text{-var},[t_0,t_1]}\leq \ltn I(x)\rtn_{q\text{-var},[t_0,t_1]}+\ltn J(x)\rtn_{q\text{-var},[t_0,t_1]}
$$	
%The last inequality  \eqref{Fx} is a direct consequence of the first two inequalities.
\eqref{Fx} holds.

(ii) By virtue of \eqref{YL1}, \eqref{gxy}  and the condition $x_{t_0} =y_{t_0}$ of the Proposition, we have
	\begin{eqnarray*}
	\ltn J(x)-J(y)\rtn_{p\text{-var},[t_0,t_{1}]}&\leq& \ltn\omega\rtn_{p\text{-var},[t_0,t_1]}\left[|g(t_0,x_{t_0}) - g(t_0,y_{t_0})|+ (K+1)\ltn g(.,x_.) - g(.,y_.)\rtn_{q_0\text{-var},[t_0,t_1]}\right]\\
	&\leq& \ltn\omega\rtn_{p\text{-var},[s,t]}(K+1)\ltn g(.,x_.) - g(.,y_.)\rtn_{q_0\text{-var},[t_0,t_1]}\\
	&\leq& (K+1)M'_N\ltn\omega\rtn_{p\text{-var},[t_0,t_1]}\|x-y\|_{q\text{-var},[t_0,t_{1}]} \left(2+\ltn x\rtn^{\delta}_{q\text{-var},[t_0,t_{1}]}+\ltn y\rtn^{\delta}_{q\text{-var},[t_0,t_{1}]}\right).
	\end{eqnarray*}
	Furthermore,
	\begin{eqnarray*}
	\left|[I(x)(t)-I(y)(t)]- [I(x)(s)-I(y)(s)]\right| &\leq &\int_s^t |f(u,x_u)-f(u,y_u)|du \\
	&\leq& L_N \int_s^t |x_u-y_u| du\\
	&\leq & L_N \|x-y\|_{\infty,[t_0,t_1]} (t-s)\\
	&\leq & L_N \|x-y\|_{q\text{-var},[t_0,t_1]} (t-s),
	\end{eqnarray*}
	hence 
	$$
	\ltn I(x)-I(y)\rtn_{q\text{-var},[t_0,t_1]} \leq M'_N \|x-y\|_{q\text{-var},[t_0,t_1]} (t_1-t_0).
	$$
Inequality \eqref{Fxy} is a direct consequence of these estimates for $I(x)$ and $J(x)$.
\end{proof}

%===================
%===========================
%\begin{proposition}\label{EstIJ1}
%Assume that ${\textbf H}_1-{\textbf H}_3$ are satisfied. 
%\end{proposition}
%=================================

%=====================================================================
Before proving the existence and uniqueness theorem, we need the following lemma of Gronwall type.
%=====================
\begin{lemma}[Gronwall-type Lemma]\label{lemma6}
Let $1\leq p\leq q$ be arbitrary and satisfy $\frac{1}{p}+\frac{1}{q}>1$.  Assume that $\omega \in \widehat{C}^p([0,T],\R)$ and $y\in \widehat{C}^{q}([0,T],\R^d)$ satisfy
\begin{equation}\label{condition1}
|y_t-y_s|\leq A_{s,t}^{1/q} + a_1\left |\int_s^t y_udu \right | + a_2\left |\int_s^t y_ud\omega_u\right |,\quad\forall s,t\in [0,T],\quad s<t, 
\end{equation}
 for some fixed control function $A$ on $\Delta[0,T]$ and some constants $a_1, a_2 \geq 0$. Then there exists a constant  $C$ independent of $T$ such that for every $s,t \in [0,T]$, $s<t$,
 \begin{eqnarray}\label{eqn.gron}
\ltn y\rtn_{q\text{-}\rm{var},[s,t]} \leq (|y_s| + A_0) e^{C(|t-s|^p+ \ltn\omega\rtn_{p\text{-}\rm{var},[s,t]}^{p})},
\end{eqnarray}
where $A_0 =A_{0,T}^{1/q}$.
\end{lemma}
%=============
\begin{proof}
Put 
$$
c:=\max\{a_1, a_2(K+1)\},
$$
in which $K$ is defined in \eqref{constK}.
We have
\begin{eqnarray*}
|y_t-y_s|&\leq & A_{s,t}^{1/q}+a_1\|y\|_{\infty,[s,t]}(t-s) + a_2\ltn \omega\rtn_{p\text{-var},[s,t]}(\|y\|_{\infty,[s,t]} + K\ltn y\rtn_{q\text{-var},[s,t]})\\
&\leq &A_{s,t}^{1/q}+\max \left\{a_1\|y\|_{\infty,[s,t]}, a_2(\|y\|_{\infty,[s,t]}+K\ltn y\rtn_{q\text{-var},[s,t]})\right\} (t-s+\ltn \omega\rtn_{p\text{-var},[s,t]}).
\end{eqnarray*}
Fix the interval $[s,t]\subset [0,T]$ and apply the above inequality for arbitrary subinterval $[u,v]\subset [s,t]$ we obtain
\begin{eqnarray*}
|y_v-y_u|&\leq & A_{u,v}^{1/q}+\max\{a_1\|y\|_{\infty,[s,t]}, a_2(\|y\|_{\infty,[s,t]}+K\ltn y\rtn_{q\text{-var},[s,t]})\} (v-u+\ltn \omega\rtn_{p\text{-var},[u,v]})\\
&\leq &A_{u,v}^{1/q}+\max\{a_1, a_2(K+1)\}(|y_s| +\ltn y\rtn_{q\text{-var},[s,t]}) (v-u+\ltn \omega\rtn_{p\text{-var},[u,v]})\\
&\leq &A_{u,v}^{1/q}+c(|y_s| +\ltn y\rtn_{q\text{-var},[s,t]}) (v-u+\ltn \omega\rtn_{p\text{-var},[u,v]}).
\end{eqnarray*}
 Therefore, by virtue of Lemma \ref{controled}, we get
\begin{eqnarray}
\ltn y\rtn_{q\text{-var},[s,t]}&\leq & A_{s,t}^{1/q}+ c(|y_s| +\ltn y\rtn_{q\text{-var},[s,t]}) (t-s+\ltn \omega\rtn_{p\text{-var},[s,t]})\notag\\
&\leq & A_0+ c(|y_s| +\ltn y\rtn_{q\text{-var},[s,t]}) (t-s+\ltn \omega\rtn_{p\text{-var},[s,t]}).\label{new1}
\end{eqnarray}
Now we construct  the sequence of greedy times $t_i$ with parameter $\{1, \frac{1}{2c}\}$ according to Subsection~\ref{subsec.greedy-times}, that is 
$$
(t_{i+1}-t_i+\ltn \omega\rtn_{p\text{-var},[t_i,t_{i+1}]})=\frac{1}{2c}.
$$
Then, by \eqref{new1} for all $s,t \in [t_i,t_{i+1}]$, $s<t$, we have
$$
\ltn y\rtn_{q\text{-var},[s,t]} \leq A_0 + \frac{1}{2}(|y_s| + \ltn y\rtn_{q\text{-var},[s,t]}),
$$
which implies
\begin{eqnarray*}
\ltn y\rtn_{q\text{-var},[t_i,t_{i+1}]}&\leq & 2A_0+ |y_{t_i}|,\\
\ltn y\rtn_{q\text{-var},[u,v]}&\leq & 2A_0+ |y_{u}|,\quad \forall \;  u,v \in [t_i,t_{i+1}],\; u<v.%\label{star}\notag
\end{eqnarray*}
Therefore,
$$
|y_{t_{i+1}}|\leq \|y\|_{\infty,[t_i,t_{i+1}]}\leq 2(A_0+|y_{t_i}|),
$$
or more generally
\begin{equation}
|y_{t_{i+1}}|\leq  \|y\|_{\infty,[s,t_{i+1}]}\leq 2(A_0+|y_{s}|), \quad \forall s\in [t_i,t_{i+1}].
\end{equation}
By induction we obtain for any $s\in [t_k,t_{k+1}]$, $0\leq k\leq i$, $i\in \{0,\dots , N(T,\omega)\}$,  where $N(T,\omega)$ is defined by \eqref{N}, the sequence of inequalities
%(Note: use another denotation to distinguish two number  $N$ for 2 different sequences of stopping time)
\begin{eqnarray}\label{induction}
2A_0+ | y_{t_{i+1}}|\leq  2(2A_0+ |y_{t_i}|)\leq \dots \leq 2^{i-k}(2A_0+|y_{t_{k+1}}|)\leq 2^{i-k+1}(2A_0+|y_{s}|).% \forall s\in [t_k,t_{k+1}],\;0\leq k\leq i
\end{eqnarray}
 %for all $i=0,\dots , N(T,\omega)$.
Hence, 
\begin{eqnarray}
\ltn y\rtn_{q\text{-var},[t_i,t_{i+1}]} \leq 2A_0+ | y_{t_{i}}|\leq 2^{i-k}(2A_0+|y_{s}|),\quad \forall s\in [t_k,t_{k+1}],\;0\leq k\leq i.
\end{eqnarray}
Now, we estimate the $q$-var norm of $y$ in an arbitrary but fixed interval $[s,t]\subset[0,T]$. Recall the sequence of greedy time defined in \eqref{stoppingtime}. If there exists $i$ such that $s<t_i<t$, put
\begin{eqnarray}
\overline{N}&:=& \sup\{n: t_n\leq t\},\notag\\
\underline{N}&:=& \inf\{n: t_n\geq s\},\label{N-barN}\\
N&:=&  \overline{N}- \underline{N}=N(s,t,\omega). \notag
\end{eqnarray}
We have $s\leq t_{\underline{N}}<t_{\underline{N}+1} <\cdots < t_{\overline{N}}\leq t$ and
\begin{eqnarray*}
\ltn y\rtn_{q\text{-var},[s,t_{\underline{N}}]}&\leq & 2A_0 + |y_s|,\\
%\ltn y\rtn_{q\text{-var},[t_i,t_{i+1}]}&\leq & 2^{N-1}( 2|A_0| + |y_{t_{\underline{N}}}|) \leq 2^{N-1}( 2|A_0| + |y_{t_{\underline{N}}}|)\\
\ltn y\rtn_{q\text{-var},[t_{\underline{N}+i},t_{\underline{N}+i+1}]}&\leq & 2^{i+1}( 2A_0 + |y_{s}|), i=0,\dots , N-1,\\
\ltn y\rtn_{q\text{-var},[t_{\overline{N}},t]}&\leq & 2A_0+|y_{t_{\overline{N}}}| \leq 2^{N}( 2A_0 + |y_{s}|).
\end{eqnarray*}
By Lemma \ref{additive} we have
 \begin{eqnarray}
\ltn y\rtn_{q\text{-var},[s,t]}&\leq& (N+1)^{\frac{q-1}{q}} \left(\ltn y\rtn^q_{q\text{-var},[s,t_{\underline{N}}]}+ \sum_{i=1}^{N-1}\ltn y\rtn^q_{q\text{-var},[t_i,t_{i+1}]} + \ltn y\rtn^q_{q\text{-var},[t_{\overline{N}},t]} \right)^{1/q}\nonumber\\
&\leq & (N+1)^{\frac{q-1}{q}}(2A_0+|y_s|)(\sum_{j=0}^{N} 2^{jq})^{1/q} \notag\\
&\leq &(N+1)(2A_0+|y_s|)2^{N}. \notag
%&\leq & (N+1)^{\frac{q-1}{q}}(2A_0+|y_s|)2^{N+1},
\end{eqnarray} 
In the case $[s,t]\subset [t_i,t_{i+1}]$ with some $i\in \{0, 1,\dots , N(T,\omega)\}$, we already have 
$$
\ltn y\rtn_{q\text{-var},[s,t]}\leq 2A_0 + |y_s|.
$$
To sum up, for any $[s,t]\subset[0,T]$ we have the estimate 
$$
\ltn y\rtn_{q\text{-var},[s,t]}\leq (2A_0+|y_s|)2^{2N}.
$$
Combining with \eqref{Nab2}, we conclude that% there exists a constant  $C_1$ independent of $T$ such that 
\begin{eqnarray*}  
\ltn y\rtn_{q\text{-var},[s,t]} &\leq& (2A_0+|y_s|)2^{4^pc^p(|t-s|^p+\ltn\omega\rtn^p_{p\text{-var},[s,t]})}\\
&\leq & (2A_0+|y_s|) e^{C(|t-s|^p+ \ltn\omega\rtn_{p\text{-}\rm{var},[s,t]}^{p})},
\end{eqnarray*}
where $C=  4^pc^p\ln 2$.
The proof is complete.
\end{proof}
%==============================
\begin{remark}
\begin{enumerate}
%\item In \cite{Martina}, the authors introduced the Rough Gronwall Lemma. In which, the role of the function $G_t$ is similar to the function $t\mapsto \ltn y\rtn_{q-\text{var},[0,t]}$ in our lemma here. Howerver, the Rough Gronwall Lemma is not applicable since its condition  is stronger than \eqref{condition1}.
\item Gronwall Lemma is an important tool in the theory of ordinary differential equations, and the theory of Young differential equations as well. Some versions of Gronwall-type lemma can be seen in \cite{nualart3} and \cite{Martina}.
\item The conclusion of Lemma \ref{lemma6} is still true if one replaces $A_0$ by $A^{1/q}_{s,t}$.
\item It can be seen from the proof that in the conditions of Lemma \ref{lemma6} we have
$$\|y\|_{\infty,[0,T]}\leq (2A_0+|y_0|)2^{N(T,\omega)+1}\leq (2A_0+|y_0|)2^{(4cT)^p+1+ (4c\ltn\omega\rtn_{p\text{-var},[s,t]})^p}.$$
%\item The constant $C_1$ in \eqref{eqn.gron} can be chosen as  $C_1 = 4^pC^p\ln 2$.
%\item Optimize the estimate by replacing $\frac{1}{2C}$ by $\mu<\frac{1}{C}$: to be continued...
\end{enumerate}
\end{remark}
%=============================
%\textcolor{red}{Restate the next corollary for $y$ and $\omega$ are vecto value}
\begin{corollary}\label{coll6}
If in Lemma~\ref{lemma6} we replace the condition \eqref{condition1} by the condition
\begin{equation}\label{condition2}
\ltn y\rtn_{q\text{-}\rm{var},[s,t]}\leq  A^{1/q}_{s,t}+ a_1(|y_s| +\ltn y\rtn_{q\text{-}\rm{var},[s,t]}) (t-s+\ltn \omega\rtn_{p\text{-}\rm{var},[s,t]})
\end{equation}
for all $s<t$ in $[0,T]$, a positive constant $a_1>0$  and $\omega \in \widehat{C}^p([0,T],\R^m)$. Then there exists a constant  $C$ independent of $T$ such that for every $s<t$ in $[0,T]$
\begin{eqnarray}
\ltn y\rtn_{q\text{-}\rm{var},[s,t]} \leq (|y_s| + A^{1/q}_{s,t})e^{C(|t-s|^p+\ltn \omega\rtn_{p\text{-}\rm{var},[s,t]}^{p})}.
\end{eqnarray}
\end{corollary}
%============================
%=========================
% dua dl chinh len day.
We are now at the position to state and prove the main theorem of this section.
\begin{theorem}[Existence and uniqueness of global solution]\label{theo2}
Consider the Young differential equation \eqref{integral.eq.Stieltes}, starting from an arbitrary initial time $t_0\in [0,T)$,
$$
	x_t=x_{t_0} +\int_{t_0}^t f(s,x_s) ds + \int_{t_0}^tg(s,x_s)d\omega_s,\quad t\in [t_0,T],\quad x_{t_0}\in\R^d.
	$$
	 with $T$ being an arbitrary fixed positive number and $x_0\in \R^d$ being an arbitrary initial condition. Assume that the conditions ${\textbf H}_1-{\textbf H}_3$ hold. Then, this equation has a unique solution $x$  in the space $\widehat{C}^{q}([t_0,T],\R^d)$, where $q$ is chosen as above satisfying \eqref{q}. Moreover, the solution is in $\widehat{C}^{p'}([t_0,T],\R^d)$, where $p'=\max\{p,\frac{1}{\alpha}\}$.  %Moreover, the solution is $\alpha_0-$H\"older continuous with $\alpha_0 = \min\{\alpha,\nu\}$.
\end{theorem}
%==================
\begin{proof}	
The proof proceeds in several steps.\\
{\bf {Step 1}}: In this step we will show the {\em local existence and uniqueness} of solution.
Set
	\begin{eqnarray}
	\mu := \dfrac{1}{2M(K+2)},\label{constmu}
	\end{eqnarray}
	where $M$ is defined in \eqref{M} and $K$ is defined in \eqref{constK}.
	Let $s_0\in [t_0,T)$ be arbitrary but fixed.
	% \textcolor{red}{ We define $t_1 \in (t_0,T]$ as the unique number such that
	%\begin{eqnarray}\label{constt1}
	%|t_{1}-t_0|^{\alpha}+\ltn \omega \rtn_{p\text{-var},[t_0,t_{1}]} = \mu;
	%\end{eqnarray}
	%in case  $|t_{1}-t_0|^{\alpha}+\ltn \omega \rtn_{p\text{-var},[t_0,t_{1}]} < \mu$ for all $t_1 \in (t_0,T]$ we set $t_1=T$.}\\
	We recall here the sequence of greedy times $\tau_n$ with the parameters $\alpha,\mu$, i.e 
      	\begin{equation*}
\tau_0 = 0,\;\;|\tau_{i+1}-\tau_i|^\alpha+\ltn\omega\rtn_{p\text{-var},[\tau_i,\tau_{i+1}]} = \mu.
\end{equation*}
Put $r_0= \min\{n: \tau_n> s_0\}$	and  define $s_1 = \min\{\tau_{r_0}, T\}$. Then, 
\begin{eqnarray}\label{constt1}
	|s_{1}-s_0|^{\alpha}+\ltn \omega \rtn_{p\text{-var},[s_0,s_{1}]} \leq \mu. 
\end{eqnarray}
	We will show that the equation \eqref{integral.eq.Stieltes} restricted to $[s_0,s_1]$,
	$$
	x_t=x_{s_0} +\int_{s_0}^t f(s,x_s) ds + \int_{s_0}^tg(s,x_s)d\omega_s,\quad t\in [s_0,s_1],\quad x_{s_0}\in\R^d,
	$$
has a a unique solution. \medskip

	{\bf{Existence of local solutions}}.\medskip
	
	Recall the mapping $F$ defined by the formula \eqref{eqn.F} with $t_0, t_1$ replaced by $s_0,s_1$, respectively. 	By Proposition \ref{EstIJ} and \eqref{constmu}--\eqref{constt1}, for $s_0, s_1$  determined above we have $F: \widehat{C}^{q}([s_0,s_1],\R^d)\longrightarrow  \widehat{C}^{q}([s_0,s_1],\R^d)$ and 
	\begin{eqnarray*}
	\|F(x)\|_{q\text{-var},[s_0,s_1]}= |F(x)_{s_0}|+\ltn F(x)\rtn_{q\text{-var},[s_0,s_1]}&\leq&|x_{s_0}|+\frac{1}{2}\left(1+\|x\|_{q\text{-var},[s_0,s_1]}\right).
	\end{eqnarray*}
	We show, furthermore, that if $x\in \widehat{C}^{q}([s_0,s_1],\R^d)$ then $F(x)\in C^{(q-\epsilon)\text{-var}}([s_0,s_1],\R^d)$ with small enough $\epsilon$. 
	Indeed, since $q>p$, $q\alpha>1$, we can choose $\epsilon > 0$ such that $q-\epsilon\geq p$ and $(q-\epsilon)\alpha \geq 1$. For all $s<t$ in $[s_0,s_1]$, using \eqref{Fx} we have 
	\begin{eqnarray*}
	|F(x)_t-F(x)_s|&\leq& \ltn F(x)\rtn_{q\text{-var},[s,t]}\\
	&\leq & M(K+2)\left(1+\|x\|_{q\text{-var},[s_0,s_1]}\right) \left((t-s)^{\alpha}+\ltn\omega\rtn_{p\text{-var},[s,t]}\right)\\
	&\leq &M(K+2)\left(1+\|x\|_{q\text{-var},[s_0,s_1]}\right)\left[\left((t-s)^{(q-\epsilon)\alpha}\right)^{\frac{1}{q-\epsilon}}+\left(\ltn \omega\rtn_{p\text{-var},[s,t]}^{(q-\epsilon)}\right)^{\frac{1}{q-\epsilon}}\right],
		\end{eqnarray*}
		then
		\begin{eqnarray*}
		\ltn F(x)\rtn_{(q-\epsilon)-\text{var},[s_0,s_1]}\leq  M(K+2)\left(1+\|x\|_{q\text{-var},[s_0,s_1]}\right) \left((s_1-s_0)^{\alpha}+\ltn\omega\rtn_{p\text{-var},[s_0,s_1]}\right)\\
		\end{eqnarray*}
		and the assertion follows by an application of 
 Lemma \ref{controled}.
	Now, looking at the mapping $F$ again, we introduce the set
$$
	B_1:= \left\{x\in \widehat{C}^{q}([s_0,s_1],\R^d)|\;x(s_0) = x_{s_0}, \|x\|_{q\text{-var},[s_0,s_1]}\leq 2|x_{s_0}|+1 \right\}.
$$
	 Taking into account \eqref{Fxy}, the map $F$ is continuous and  
	 $$	 
	 F:B_1 \to B_1.	 
	 $$
	We show that $B_1$ is a closed convex set in the Banach space $\widehat{C}^{q}([s_0,s_1],\R^d)$,  and $F$ is a compact operator on $B_1$. Indeed, for the former observation, note that if $z=\lambda x+(1-\lambda)y$ for some $x,y\in B_1, \lambda\in [0,1]$ then
	$$
z_{s_0} = \lambda x_{s_0}+(1-\lambda)y_{s_0}= \lambda x_{s_0}+(1-\lambda)x_{s_0} = x_{s_0}
$$	
and 
$$
\|z\|_{q\text{-var},[s_0,s_1]} = \|\lambda x+(1-\lambda)y\|_{q\text{-var},[s_0,s_1]}\leq \lambda\|x\|_{q\text{-var},[s_0,s_1]}+(1-\lambda)\|y\|_{q\text{-var},[s_0,s_1]} \leq 2|x_{s_0}|+1.
$$
Now,  we prove that for any sequence $y^n\in F(B_1)$, there exists an subsequence converges in $p\text{-var}$ norm to an element $y\in B_1$, i.e. $F(B_1)$ is relatively compact in $B_1$. To do that, we will show that $(y^n)$ are   equicontinuous, bounded in $(q-\epsilon)\text{-var}$ norm. Namely, take the sequence $y^n=F(x^n)\in F(S)$, $x^n\in B_1$. Then, by virtue of Lemma \ref{controled} we have% {\bf Hong, please verify!}
$$
\sup_n \|y^n\|_{(q-\epsilon)\text{-var},[s_0,s_1]}\leq |x_{s_0}|+ 2M(K+2)(1+|x_{s_0}|)((s_1-s_0)^{\alpha}+\ltn\omega\rtn_{p\text{-var},[s_0,s_1]}).
$$
It means that $y^n$ are bounded in $C([s_0,s_1],\R^d)$ with sup norm, as well as bounded in $C^{(q-\epsilon)\text{-var}}([s_0,s_1],\R^d)$.\\
Moreover, for all $n$, $s_0\leq s\leq t\leq s_1$,
 $$
 |y^n_t-y^n_s|\leq 2M(K+2)(1+|x_{s_0}|)\left((t-s)^{\alpha}+\ltn\omega\rtn_{p\text{-var},[s,t]}\right),
 $$
which implies that $(y^n)$ is equicontinuous. Applying Proposition 5.28 of \cite{friz}, we conclude that $y^n$ converges to some $y$ along a subsequence in $\widehat{C}^{q}([s_0,s_1],\R^d)$. This proves the compactness of $\overline{F(B_1)}$. Hence, $F(B_1)$ is a relative compact set in $\widehat{C}^{q}([s_0,s_1],\R^d)$. 
We conclude that $F$ is a compact operator from $B_1$ into itself. 
Therefore, by the Schauder-Tychonoff fixed point  theorem (see e.g \cite[Theorem 2.A, p. 56]{Zeidler}), there exists a function $\hat{x}\in B_1$ such that $F(\hat{x})=\hat{x}$, thus there exists a solution  $\hat{x}\in B_1$ of \eqref{integral.eq.Stieltes} on the interval $[s_0,s_1]$.
	%}
	%%%%%%%%%%%%%%%%%%%%%%%%%%%%%%%%%5
\medskip\\
	{\bf {Uniqueness of local solutions}.}
	\medskip\\
		Now, we assume that $x,y$ are two solutions in $\widehat{C}^{q}([s_0,s_1],\R^d)$ of the equation \eqref{integral.eq.Stieltes}  such that $x_{s_0}=y_{s_0}$. It follows that
	$F(x)=x$ and $F(y)=y$. Put
	$$
	N_0 = \max\{\|x\|_{q\text{-var},[s_0,s_1]},\|y\|_{q\text{-var},[s_0,s_1]}\},
	$$ 
	and $z=x-y$, we have $z_{s_0} = 0$ and
	$$\|x\|_{\infty,[s_0,s_1]},\|y\|_{\infty,[s_0,s_1]}\leq N_0.$$
		By virtue of Proposition \ref{EstIJ}(ii), we obtain
	\begin{eqnarray}
	\ltn z\rtn_{q\text{-var},[s,t]}&=&\ltn x-y\rtn_{q\text{-var},[s,t]}= \ltn Fx-Fy\rtn_{q\text{-var},[s,t]}\\
	&\leq& M'_{N_0}(K+1) (1+ 2N_0^\delta) \left(|z_s|+ \ltn z\rtn_{q\text{-var},[s,t]}\right)(t-s+\ltn\omega\rtn_{p\text{-var},[s,t]}).\notag
	\end{eqnarray}
Applying Corollary \ref{coll6} to the function $z$, since $z_{s_0} = 0$ we conclude that $\ltn z\rtn_{q\text{-var},[s_0,s_1]} =0$. That  implies $z\equiv 0$ on $[s_0,s_1]$. The uniqueness of the local solution is proved.

{\bf{Step 2}}: Next, by virtue of the additivity of the Riemann and Young integrals, the solution can be concatenated. Namely,  let $0<t_1<t_2<t_3\leq T$. Let $x_t$ be a solution of the equation
$$
  x_t = x_{t_1} + \int_{t_1}^t f(s,x_s) ds +  \int_{t_1}^t g(s,x_s) d\omega_s, \quad t\in [t_1,t_2],
$$
  and $y_t$ be a solution of the equation
  $$
   y_t = y_{t_2} + \int_{t_2}^t f(s,y_s) ds + \ \int_{t_1}^t g(s,y_s) d\omega_s, \quad t\in [t_2,t_3], 
      $$
 and $y(t_2)=x(t_2)$. Define a continuous function $z(\cdot) : [t_1,t_3] \rightarrow\R^d$ by setting $z(t)=x(t)$ on $[t_1,t_2]$ and $z(t)=y(t)$ on $[t_2,t_3]$. Then $z(\cdot)$ is the solution of the Young differential equation
 $$
  z_t = x_1 + \int_{t_1}^t f(s,z_s) ds +  \int_{t_1}^t g(s,z_s) d\omega_s, \quad t\in [t_1,t_3].
 $$
 Conversely, If $z_t$ is a solution on $[t_1,t_3]$ then its restrictions on $[t_1,t_2]$ and on $[t_2,t_3]$ are solutions of the corresponding equation with the corresponding initial values.

{\bf Step 3}: Finally, apply the estimates \eqref{Nab2} to the case of $\mu$ being defined by \eqref{constmu}, we can easily get the unique global solutions to the equation \eqref{integral.eq.Stieltes} on $[t_0,T]$.

Put $n_0= \min\{n: \tau_n> t_0 \}$.	The interval $[t_0,T]$ can be covered by $N(T,\omega)-n_0+1$ intervals $[t_i,t_{i+1}]$, $i=\overline{0,N(T,\omega)-n_0+1}$, determined by greedy times $t_i=\tau_{n_0+i-1}$, $i=1,\ldots, N(T,\omega)-n_0$,  with parameter $\mu$ being defined by \eqref{constmu} and $t_{N(T,\omega)+1} :=T$. The arguments in {\bf{Step 1}} are applicable to each of intervals $[t_i,t_{i+1}]$, $i=\overline{0,N(T,\omega)}$, implying the existence and uniqueness of solutions on those intervals. Then, starting at $x(t_0) = x_{t_0}$ the unique solution of \eqref{integral.eq.Stieltes} on $[t_0,t_1]$ is extended uniquely to $[t_1,t_2]$, then further by induction up to  $[t_{N(T,\omega)-1},t_{N(T,\omega)}]$ and lastly to $[t_{N(T,\omega)},T]$. The solution $x$ of \eqref{integral.eq.Stieltes} on $[t_0,T]$ then exists uniquely.
	
	Furthermore, for all $\epsilon$ such that $q-\epsilon\geq p'$ the solution $x$ belongs to $\widehat{C}^{q-\epsilon}([t_i,t_{i+1}],\R^d)$, for all $i=\overline{0,N(T,\omega)}$. Hence, 
	$x\in \widehat{C}^{p'}([t_0,T],\R^d)$.
\end{proof}

%=============================================

\begin{proposition}\label{growth}
Assume that the conditions ${\textbf H}_1-{\textbf H}_3$ are satisfied. Let $0\leq t_0<T$. Denote by $x(\cdot)=x(t_0,\cdot,\omega,x_{0})$ the solution of the equation \eqref{integral.eq.Stieltes} on $[t_0,T]$. Then  there exist positive constants $C_1=C_1(T)$, $C_2=C_2(T)$ such that
%\begin{eqnarray}\label{EstSolution}
%\ltn x\rtn_{q\text{-}\rm{var},[s,t]}\leq C_1 (|x_s|+|t-s|^\alpha+\ltn \omega\rtn_{p\text{-}\rm{var},[s,t]})e^{C_2\|\omega\|^{p'}_{p\text{-}\rm{var},[s,t]}},\quad \forall s<t\in [t_0,T],
%\end{eqnarray}
%In particular,
\begin{eqnarray}\label{EstSolution0}
\| x\|_{q\text{-}\rm{var},[t_0,T]}\leq C_1[1+(T-t_0)^\alpha] (1+|x_0|)(1+\ltn\omega\rtn_{p\text{-}\rm{var},[t_0,T]})e^{C_2\ltn\omega\rtn^{p'}_{p\text{-}\rm{var},[t_0,T]}},
\end{eqnarray}
where $p'=\max\{p,\frac{1}{\alpha}\}$.
\end{proposition}
%=============================================
\begin{proof}
Since $x$ is a solution, $x=Fx$, hence
by \eqref{Fx} we have
\begin{eqnarray*}
\ltn x\rtn_{q\text{-var},[s,t]}&\leq& M(K+2)\left(1+\|x\|_{q\text{-var},[s,t]}\right) \left((t-s)^{\alpha}+\ltn\omega\rtn_{p\text{-var},[s,t]}
\right)\\
&\leq & M(K+2)\left((t-s)^{\alpha}+\ltn\omega\rtn_{p\text{-var},[s,t]}
\right)\\
&& + M(K+2) (|x_s|+\ltn x\rtn_{q\text{-var},[s,t]})\left((t-s)^{\alpha}+\ltn\omega\rtn_{p\text{-var},[s,t]}
\right)
\end{eqnarray*}
Use the arguments similar to that of the proof of Lemma \ref{lemma6}
we conclude that there exist $C_1=C_1(T)$ and $C_2=C_2(T)$ such that
$$\ltn x\rtn_{q\text{-var},[s,t]}\leq C_1 (|x_s|+|t-s|^\alpha+\ltn \omega\rtn_{p\text{-var},[s,t]})e^{C_2\ltn\omega\rtn^{p'}_{p\text{-var},[s,t]}},\quad \forall s<t\in [t_0,T].
$$
Thus, we get \eqref{EstSolution0}.%\textcolor{blue}{ with the constants above can be choosen as
%$$
%C_1 = M(K+2)2^{4^p(K+2)^pM^p(T-t_0)^p},\;\;C_2 = 2^{4^p(K+2)^pM^p}
%$$%, and then \eqref{EstSolution0}.
%}
%\textcolor{red}{In the autonomuous case, M is independent with T, look at the above formula we can see how the constants depend on T and other parameters, do we need a remark here?}
\end{proof}
%===============================
\medskip

In order to study the flow generated by the solution of system \eqref{integral.eq.Stieltes} in the next section, we need also to consider the backward version of \eqref{integral.eq.Stieltes} in the following form
  \begin{equation}\label{RDEbw}
 x_t = x_T + \int_t^T f(s,x_s) ds +  \int_t^T  g(s,x_s) d\omega_s, \quad t\in [0,T],
 \end{equation}
 where $x_T\in\R^d$ is the initial value of the backward equation \eqref{RDEbw}, the coefficient functions $f: [0,T] \times \R^d \rightarrow \R^d$,  
 $g: [0,T] \times \R^d \rightarrow \R^d\times \R^m$ are continuous functions, and the driven force $\omega: [0,T] \rightarrow \R^m$ belongs to $\widehat{C}^{p}([0,T],\R^m)$.
 
  \begin{theorem}[Existence and uniqueness of solutions of backward equation]\label{theo3}
   Consider the backward equation
  \eqref{RDEbw} on $[0,T]$. 
   Assume that the conditions ${\textbf H}_1-{\textbf H}_3$ hold.
  Then the backward equation \eqref{RDEbw} has a unique solution $x\in \widehat{C}^{q}([0,T],\R^d)$,  where $q$ is chosen as above satisfying \eqref{q}. 
 % Moreover the solution is $\alpha_0$-H\"older continuous.
  \end{theorem}
\begin{proof}
We make a change of variables
$$
{\hat f}(u,x) := f(T-u,x), \quad {\hat g}(u,x) := g(T-u,x),\quad 
{\hat \omega}(u) := \omega(T-u), \quad y_u := x_{T-u},\quad u\in [0,T].
$$
Then $x_T = y_0$, and by putting $v=T-t$ and $u=T-s$ we have 
$$
\int_t^T f(s,x_s) ds = \int_t^T f(T-u,x_{T-u}) ds = -\int_{T-t}^0 {\hat f}(u,y_u) du = \int_0^v {\hat f}(u,y_u) du.
$$
Furthermore,  by virtue of the property \eqref{Y-integral-direction}
 of the Young integral  we have
$$
\int_t^T g(s,x_s) d\omega_s
= \int_t^T  g(T-u,x_{T-u}) d\omega_{T-u} 
= \int_0^v {\hat g}(u,y_u) d{\hat \omega}_u.
$$
Therefore, the backward equation   \eqref{RDEbw} is equivalent to the forward equation
\begin{equation}\label{integral.eq.Stieltes.foward1}
 y_v = y_0 + \int_0^v {\hat f}(u,y_u) du +  \int_0^v {\hat g}(u,y_u) d{\hat \omega}_u, \quad v\in [0,T],
 \end{equation}
  where $y_0=x_T\in\R^d$. 
    Now, we verify the conditions of Theorem \ref{theo2} for the forward equation \eqref{integral.eq.Stieltes.foward1}. First note that if $\omega\in \widehat{C}^{p}([0,T],\R^m)$   then ${\hat \omega}\in \widehat{C}^{p}([0,T],\R^m)$. Furthermore, 
 the condition (${\textbf H}_1$) obviously holds for $\hat g$ and the condition (i) of (${\textbf H}_2$)  holds for $\hat f$. For the condition (ii) of (${\textbf H}_2$) we note that if it holds for $f$ then
  $$
  |{\hat f}(v,x)| = |f(T-v,x)| \leq a|x| + b(T-v) = a|x| + {\hat b}(v), \quad v\in [0,T],
  $$
  where 
${\hat b}(v) = b(T-v) \in L_\frac{1}{1- \alpha}([0,T],\R^d)$ because (${\textbf H}_2$)(ii) is satisfied for $f$. Thus,
(${\textbf H}_2$)(ii) is satisfied for $\hat f$. Consequently,  Theorem \ref{theo2} is applicable to the forward equation \eqref{integral.eq.Stieltes.foward1}  implying that \eqref{integral.eq.Stieltes.foward1}  has unique solution
$y\in \widehat{C}^{q}([0,T],\R^d)$. 
Since \eqref{integral.eq.Stieltes.foward1} is equivalent to 
the backward equation   \eqref{RDEbw} we have the theorem proved.
\end{proof}
%================================= 
%The next theorem concludes that the solution is in fact continuous w.r.t. initial conditions.\\
%{\textcolor{red}{move to section 4?(using the flow property in the proof)}}
\begin{theorem}\label{contsolution}
Suppose that the assumptions of  Theorem \ref{theo2} are satisfied. Denote by $X(t_0,\cdot,\omega,x_0)$ the solution of \eqref{integral.eq.Stieltes} starting from $x_0$ at time $t_0$, i.e. $X(t_0,t_0,\omega,x_0)=x_0$. Then the solution mapping 
\begin{eqnarray*}
X: [0,T]\times[0,T]\times \widehat{C}^{p}([0,T],\R^m)\times \R^d &\rightarrow& \R^d,\\
(s,t,\omega,z) &\mapsto& X(s,t,\omega,z),
\end{eqnarray*}
 is continuous.% w.r.t. $(t_1,t_2,\omega,x_0)$.
\end{theorem}
%================================
\begin{proof}
 First observe that, fixing $(\omega,x_0) \in \widehat{C}^{p}([0,T],\R^m)\times \R^d $ and looking at forward and backward equations \eqref{integral.eq.Stieltes} and \eqref{RDEbw}, we can extend the solution $X(t_0,\cdot,\omega,x_0)$ of \eqref{integral.eq.Stieltes}, with the initial value $x_0$ at $t_0$ to the whole $[0,T]$.
 The proof is divided into several steps. \medskip

 {\bf {Step 1 (Continuity w.r.t $x_0$)}:}\medskip
 
 By Proposition~\ref{growth}, we can choose $N_0$ (depending on $x_0$, $\omega$) such that
 $$
 \|X(t_0,\cdot,\omega',x'_0)\|_{q\text{-var},[0,T]}\leq N_0
 $$
for all $t_0\in [0,T]$, $|x_0-x'_0|\leq 1, \|\omega-\omega'\|_{p\text{-var},[0,T]}\leq 1$.
We use here, for short, notation
 $y_. = X(t_0,\cdot,\omega',x_0)$, $y'_. = X(t_0,\cdot,\omega',x'_0)$.
Using arguments similar to that of the proof of Proposition \ref{EstIJ}(ii), we have
\begin{eqnarray*}
|(y-y')_t-(y-y')_s|&\leq & \int_s^t|f(u,y_u)-f(u,y'_u)|du+ \left|\int_s^t g(u,y_u)-g(u,y'_u)d\omega'_u\right|\\
				&\leq & M_{N_0}^\prime (t-s)\|y-y'\|_{\infty,[s,t]} \\
				&&+  M_{N_0}^\prime(K+1)\ltn \omega'\rtn_{p\text{-var},[s,t]}\left( |y_s-y'_s| + \ltn y-y'\rtn_{q\text{-var},[s,t]}\right)(2+2N_0^\delta)\\
				&\leq &  M_{N_0}^\prime(K+1)(2+2N_0^\delta) ( |y_s-y'_s| + \ltn y-y'\rtn_{q\text{-var},[s,t]}) \left( t-s+ \ltn \omega'\rtn_{p\text{-var},[s,t]}\right).
\end{eqnarray*} 
Due to Corollary \ref{coll6}, there exist constants $C_3,C_4$ depending on parameters of the equation \eqref{integral.eq.Stieltes} and $N_0$, such that
$$
\ltn y-y'\rtn_{q\text{-var},[0,T]}\leq |y_0-y'_0| C_3 e^{C_4\ltn \omega'\rtn_{p\text{-var},[0,T]}^p}\leq |y_0-y'_0| C_3 e^{C_4(1+\| \omega\|_{p\text{-var},[0,T]})^p}.
$$
Therefore, 
\begin{eqnarray*}
|x_t-y_t| &\leq &|x_{t_0}-y_{t_0}| +\ltn x-y\rtn_{q\text{-var},[t_0,t]}\\
&\leq & |x_0-x'_0| \left(C_3 e^{C_4(1+\| \omega\|_{p\text{-var},[0,T]})^p} +1\right).
\end{eqnarray*} 
Consequently, we find a positive constants $C_1(T,\omega,x_0)$ such that  for all $t_0,t\in [0,T]$, all $\omega'$ such that $\|\omega'-\omega\|_{p\text{-var},[0,T]}<1$, we have
\begin{eqnarray}\label{cont1}
|X(t_0,t,\omega',x_0')-X(t_0,t,\omega',x_0)|\leq C_1(T,\omega,x_0)|x_0-x'_0|.
\end{eqnarray}
% It mean that the $X()$ is continuous w.r.t the initial condition.\\
%=======================================

{\bf Step 2 (Continuity w.r.t. $\omega$):} \medskip

%In this step, we will prove that with fixed $x_0$ the map $X(s,t,.,x_0)$ is continuous for all $s,t\in [0,T]$.
%In the following arguments, $C$ represent some constants. \\
Let $\omega'\in \widehat{C}^{p}([0,T],\R^m) $ be such that $\|\omega'-\omega\|_{p\text{-var},[0,T]}\leq 1$.  We use here, for short, notation  $x_.=X(t_0,\cdot,\omega, x_0)$, $x'_.=X(t_0,\cdot,\omega', x_0)$. For all $s<t$ in $[0,T]$, we have
\begin{eqnarray*}
x_t-x_s=\int_s^tf(u,x_u)du+\int_s^t g(u,x_u)d\omega_u,\\
x^\prime_t-x^\prime_s=\int_s^tf(u,x'_u)du+\int_s^t g(u,x'_u)d\omega'_u.
\end{eqnarray*}
This implies
\begin{eqnarray*}
|(x'-x)_t-(x'-x)_s|&=&\left|\int_s^t[f(u,x'_u)-f(u,x_u)]du+\int_s^t [g(u,x'_u)-g(u,x_u)]d\omega_u\right .\\
&&\left .+\int_s^t g(u,x'_u)d(\omega'-\omega)_u\right|\\
&\leq &L_{N_0}(t-s)\|x'-x\|_{\infty,[s,t]} +M(K+1)(1+\|x'\|_{q\text{-var},[s,t]})\ltn\omega'-\omega\rtn_{p\text{-var},[s,t]}\\
&& +\ltn\omega'\rtn_{p\text{-var},[s,t]}M_{N_0}'(K+1) \left(|(x'-x)_s|+\ltn x'-x\rtn_{q\text{-var},[s,t]}\right)\\
&&\times\left(2+\ltn x'\rtn^\delta_{q\text{-var},[0,T]}+\ltn x\rtn^\delta_{q\text{-var},[0,T]}\right)\\
%\hspace*{-3cm} &&\\
&\leq& C_5\ltn\omega'-\omega\rtn_{p\text{-var},[s,t]} \\
&&+ C_6\left(t-s+\ltn\omega'\rtn_{p\text{-var},[s,t]}\right)\left( |(x'-x)_s|+\ltn x'-x\rtn_{q\text{-var},[s,t]}\right)\\
&\leq& C_5\ltn\omega'-\omega\rtn_{p\text{-var},[s,t]} \\
&&+  C_6\left(t-s+\ltn\omega'\rtn_{p\text{-var},[s,t]}\right)\left( |(x'-x)_s|+\ltn x'-x\rtn_{q\text{-var},[s,t]}\right),
\end{eqnarray*}
where $C_5,C_6$ depend on $N_0$. Consequently, by virtue of 
  Lemma \ref{controled} we get 
\begin{eqnarray*}
\ltn x'-x\rtn_{q\text{-var},[s,t]}&\leq& C_3\ltn\omega'-\omega\rtn_{p\text{-var},[s,t]} \\
&&+ C_4\left(t-s+\ltn\omega\rtn_{p\text{-var},[s,t]}\right)\left( |(x'-x)_s|+\ltn x'-x\rtn_{q\text{-var},[s,t]}\right).
\end{eqnarray*}
Now, since $x'_{t_0}-x_{t_0}=0$, using Collorary \ref{coll6}  on $[t_0,t]$ (or $[t,t_0]$ and use backward equation if $t<t_0$) we find positive constant $ C_2(T,\omega,x_0)$ such that %(???$C'(T,\omega,x_0),$)
$$
\ltn x'-x\rtn_{q\text{-var},[t_0,t]}\leq C_2(T,\omega,x_0)\ltn\omega'-\omega\rtn_{p\text{-var},[t_0,t]} 
\leq C_2(T,\omega,x_0)\|\omega'-\omega\|_{p\text{-var},[0,T]}.
$$
Therefore, for all $t_0,t\in [0,T]$,
\begin{eqnarray}\label{cont2}
| X(t_0,t,\omega',x_0)-X(t_0,t,\omega,x_0)|\leq C_2(T,\omega,x_0)\|\omega'-\omega\|_{p\text{-var},[0,T]}.
\end{eqnarray}
%Note that we can choose a common constant $C(T,\omega,x_0)$ for both  \eqref{cont1}and  \eqref{cont2}.\medskip
%++=================================

{\bf Step 3 (Continuity in all variables):} \medskip

Now we fix $(t_1,t_2,\omega,x_0)$ and let $(t_1^\prime,t_2^\prime,\omega^\prime,x_0^\prime)$ be in a neighborhood of $(t_1,t_2,\omega,x_0)$  such that 
$$
|t_1-t_1^\prime|, |t_2-t_2^\prime|, \|\omega-\omega^\prime\|_{p\text{-var},[0,T]}, |x_0-x_0^\prime|\leq 1.
$$
By triangle inequality and \eqref{cont1}, \eqref{cont2}, we have
\begin{eqnarray*}
%|X(t_1^\prime, t_2^\prime,\omega^\prime,x_0^\prime) - X(t_1,t_2,\omega,x_0)| &\leq&  |X(t_1^\prime,t_2^\prime,\omega^\prime,x_0^\prime) - X(t_1^\prime,t_2^\prime,\omega,x_0)| +|X(t_1^\prime,t_2^\prime,\omega,x_0) - X(t_1,t_2,\omega,x_0)| \\
|X(t_1^\prime, t_2^\prime,\omega^\prime,x_0^\prime) - X(t_1,t_2,\omega,x_0)|&\leq & |X(t_1^\prime,t_2^\prime,\omega^\prime,x_0^\prime) - X(t_1^\prime,t_2^\prime,\omega',x_0)|+ |X(t_1^\prime,t_2^\prime,\omega^\prime,x_0) - X(t_1^\prime,t_2^\prime,\omega,x_0)|\\
&&+|X(t_1^\prime,t_2^\prime,\omega,x_0) - X(t_1,t_2',\omega,x_0)|+ |X(t_1,t_2',\omega,x_0)-X(t_1,t_2,\omega,x_0)|\\
&\leq & (C_1(T,\omega,x_0) + C_2(T,\omega,x_0))(|x'_0-x_0|+\|\omega'-\omega\|_{p\text{-var},[0,T]})\\
&&\hspace*{-2cm}+ |X(t_1',t_2',\omega,x_0)- X(t_1^\prime,t_2^\prime,\omega,X(t_1,t_1',\omega,x_0))| + \ltn X(t_1,\cdot,\omega,x_0)\rtn_{q\text{-var},[t_2,t_2']}
\end{eqnarray*} 
It is obvious that when the triple $(|x'_0-x_0|,\|\omega'-\omega\|_{p\text{-var},[0,T]}, |t_2'-t_2|)$ tends to 0 we have 
  $(C_1(T,\omega,x_0) + C_2(T,\omega,x_0))(|x-x_0|+\|\omega'-\omega\|_{p\text{-var},[0,T]})\to 0$ and $\ltn X(t_1,.,\omega,x_0)\rtn_{q\text{-var},[t_2,t_2']}\to 0$. 
As for the remaining term, let $|t_1'-t_1|$ be small enough so that $|X(t_1,t_1',\omega,x_0)-x_0|\leq 1$, using \eqref{cont1} again we obtain
\begin{eqnarray*}
 |X(t_1',t_2',\omega,X(t_1,t_1',\omega,x_0))- X(t_1^\prime,t_2^\prime,\omega,x_0)|&\leq & C_1(T,\omega,x_0) |X(t_1,t_1',\omega,x_0))-x_0|\\
&\leq & C_1(T,\omega,x_0)\ltn X(t_1,\cdot, \omega,x_0)\rtn_{q\text{-var},[t_1,t_1']},
\end{eqnarray*}
hence $ |X(t_1',t_2',\omega,X(t_1,t_1',\omega,x_0))- X(t_1^\prime,t_2^\prime,\omega,x_0)| \to 0$ as $|t_1'-t_1|\to 0$. Summing up the above arguments, we conclude that $X$ is continuous.
\end{proof}
%\end{proof}
%================================= 
\begin{remark}
 The time interval in Theorem~\ref{theo2} to Theorem \ref{contsolution} needs not be $[0,T]$. It can be $[t_0,t_0+T]$ for any $t_0\in\R$, $T>0$.
\end{remark}
%=================================    
\section{Topological flow generated by Young differential equations}
In this section we show that Young differential equations have many properties of ordinary differential equations. Especially, their solutions generate a two-parameter flow on the phase space $\R^d$, thus we can study the long term behavior of the solution flow using the tools of the theory of dynamical systems. Moreover, by defining appropriate dynamics in the space of  functions $\omega$ in $\widetilde{C}^{p}(\R,\R^m)$, we can study the long term behavior of the flow also in term of dynamics of $\omega$. For simplity of the presentation, we will assume from now on that all hypotheses  ${\textbf H}_1-{\textbf H}_3$ hold for all $T>0$ where all the parameters are independent of $T$.  
 
\subsection{Topological two-parameter flows for nonautonomous systems}

 \begin{theorem}[Different trajectories do not intersect]\label{flow1}
Assume that the conditions ${\textbf H}_1-{\textbf H}_3$ hold.
 Let $x_t$ and $\hat{x}_t$ be two solutions of the Young differential equation \eqref{integral.eq.Stieltes} on $[0,T]$. If  $x_a = \hat{x}_a$ for some $a\in [0,T]$ then $x_t =\hat{x}_t$ for all $t\in [0,T]$. In other words, two solutions of the differential equation \eqref{integral.eq.Stieltes} either coincide or do not intersect.
 \end{theorem}
 \begin{proof}
Suppose that $x_a=\hat{x}_a$ for some $a\in [0,T]$. If $a=0$ then by the uniqueness of the solution provided by Theorem~\ref{theo2}, $x_t =\hat{x}_t$ for all $t\in [0,T]$. Let $a\in (0,T]$.
Since the restrictions of the functions $x_t$ and $\hat{x}_t$ on $[a,T]$ are solutions of the equation
$$
 x_t = x_a + \int_0^t f(s,x_s) ds +  \int_0^t g(s,x_s) d\omega_s, \quad t\in [a,T],
$$
with the initial value $x_a=\hat{x}_a$,  Theorem~\ref{theo2} implies that  $x_t=\hat{x}_t$ for all $t\in [a,T]$.

Now, consider the restrictions of the functions $x_t$ and $\hat{x}_t$ on $[0,a]$. They are solutions of the equations
 $$
 x_t = x_0 + \int_0^t f(s,x_s) ds + \ \int_0^t g(s,x_s) d\omega_s, \quad t\in [0,a],
$$
with the initial values $x_0$ and ${\hat x}_0$ respectively. Since  $x_a=\hat{x}_a$ we have
\begin{eqnarray*}
x_0 + \int_0^a f(s,x_s) ds + \ \int_0^a g(s,x_s) d\omega_s &=& x_a \; =\;  {\hat x}_a\\
&=& {\hat x}_0 + \int_0^a f(s,{\hat x}_s) ds +  \int_0^a g(s,{\hat x}_s) d\omega_s.
\end{eqnarray*}
Hence,
\begin{eqnarray*}
x_0 &=& x_a - \int_0^a f(s,x_s) ds -  \int_0^a g(s,x_s) d\omega_s,\\
{\hat x}_0 &=& x_a- \int_0^a f(s,{\hat x}_s) ds -  \int_0^a g(s,{\hat x}_s) d\omega_s.
\end{eqnarray*}
 Therefore, on $[0,a]$ the two functions $x_t$ and $\hat{x}_t$ are solutions of the same backward equation
 \begin{equation}\label{eqn.nonintersection.bw}
  x_t = x_a - \int_t^a f(s,x_s) ds -  \int_t^a g(s,x_s) d\omega_s, \quad t\in [0,a],
\end{equation}
 with the same initial value $x_a$. Clearly, Theorem \ref{theo3} is applicable and  provides uniqueness of solution of the backward equation \eqref{eqn.nonintersection.bw} on $[0,a]$, hence  $x_t$ must coincide with $\hat{x}_t$  on $[0,a]$ and the theorem is proved.
 \end{proof}
 
 \begin{remark}[Locality of Young differential equations]
 By virtue of Theorems \ref{theo2}, \ref{theo3} and \ref{flow1}, under the assumptions of Theorem \ref{theo2}, the equation \eqref{integral.eq.Stieltes}  has locality properties like ODE: we can solve it locally and extend the solution both forward and backward, and any two solutions meeting each other at some time should coincide in the common interval of definitions.
 \end{remark}

 Now, in analog with the theory of ordinary differential equation we give a definition of the Cauchy operator of the equation \eqref{integral.eq.Stieltes}, which is an operator in $\R^d$ acting along trajectoties of \eqref{integral.eq.Stieltes}. 
  
  \begin{definition}[Cauchy operator]\label{dfn.Cauchy}
%Suppose that the assumptions of Theorem \ref{theo2} are satisfied.
Suppose that the conditions ${\textbf H}_1-{\textbf H}_3$ hold. 
 For any $-\infty< t_1\leq t_2<+\infty$, any $\omega\in \widetilde{C}^p(\R,\R^m)$ the {\em Cauchy operator} $X(t_1,t_2,\omega,\cdot)$ of the equation \eqref{integral.eq.Stieltes} is defined as follows:
 $$
 X(t_1,t_2,\omega,\cdot) : \R^d \rightarrow \R^d
 $$ 
 is the mapping along trajectories of \eqref{integral.eq.Stieltes}  from time moment $t_1$ to time moment $t_2$, i.e., for any vector $x_{t_1}\in \R^d$ we define $X(t_1,t_2,\omega,x_{t_1})$ to be the vector $x_{t_2}\in\R^d$ which is the value of the solution $x$ of the equation
  $$
   x_t = x_{t_1} + \int_{t_1}^t f(s,x_s) ds +  \int_{t_1}^t g(s,x_s) d\omega_s, \quad t\in [t_1,t_2],  
$$
 evaluated at time $t_2$.
 \end{definition}
 
 \begin{theorem}\label{cauchy.operator}
% Suppose that the assumptions of Theorem \ref{theo2} are satisfied.
Assume that the conditions ${\textbf H}_1-{\textbf H}_3$ hold. 
 For any $-\infty< t_1\leq t_2<+\infty$ the Cauchy operator $X(t_1,t_2,\omega,\cdot)$ of  \eqref{integral.eq.Stieltes}  is a homeomorphism. Moreover, $X(t_1,t_2,\omega,\cdot)=id$.
 \end{theorem}
 \begin{proof}
 By Theorem~\ref{flow1} the Cauchy operator $X(t_1,t_2,\omega,\cdot)$ is an injection. Using arguments of the proof of Theorem~\ref{flow1} we get that the equation
 \begin{equation}\label{eqn.fw}
   x_t = x_{t_1} + \int_{t_1}^t f(s,x_s) ds +  \int_{t_1}^t g(s,x_s) d\omega_s, \quad t\in [t_1,t_2],   
\end{equation}
 with the terminal value $x_{t_2}\in\R^d$ and unknown initial value $x_{t_1}$,
 is equivalent to the following initial value problem for the backward equation on $[t_1,t_2]$
  \begin{equation}\label{eqn.bw}
   x_t = x_{t_2} - \int_t^{t_2} f(s,x_s) ds -  \int_t^{t_2} g(s,x_s) d\omega_s, \quad t\in [t_1,t_2],
\end{equation}
with initial value $x_{t_2}\in\R^d$,
  hence Theorem \ref{theo3} is applicable and provides existence of solution for any terminal value $x_{t_2}$ of the forward equation on $[t_1,t_2]$. Consequently, the Cauchy operator $X(t_1,t_2,\omega,\cdot)$ is a surjection, thus a bijection.
  
 It is clear from the proof of Theorem \ref{theo2} and Theorem \ref{contsolution} that the solutions of \eqref{integral.eq.Stieltes} depend continuously on the initial values. Therefore, the  Cauchy operator $X(t_1,t_2,\omega,\cdot)$ acts continuously on $\R^d$.  Similar conclusion holds for the inverse $X^{-1}(t_1,t_2,\omega,\cdot)$ by using backward equation. Hence $X(t_1,t_2,\omega,\cdot)$  is a homeomorphism and trivially $X(t_1,t_1,\omega,\cdot)=id$.
 \end{proof}
 %============================
 \medskip
 Following \cite[page 114]{kunita}, below we introduce the concept of two parameter flows.
 \begin{definition}[Two-parameter flow]\label{dfn.2flow}
 \rm
 A family of mappings $X_{s,t} : \R^d \rightarrow \R^d$ depending on two real \text{var}iables $s,t\in [a,b] \subset \R$ is call a {\em two-parameter flow of homeomorphisms of $\R^d$ on $[a,b]$} if it satisfies the following conditions:\smallskip\\
 (i) For any $s,t\in [a,b]$ the mapping $X_{s,t}$ is a homeomorphism of $\R^d$;\smallskip\\
 (ii) $X_{s,s} = id$ for any $s\in [a,b]$;\smallskip\\
 (iii) $X_{s,t}^{-1} = X_{t,s}$ for any $s,t\in [a,b]$;\smallskip\\
 (iv) $X_{s,t} = X_{u,t}\circ X_{s,u}$ for any $s,t,u\in [a,b]$.%\smallskip\\
% If, additionally, $X_{s,t}$ is a linear operator of $\R^d$ for any $s,t\in [a,b]$ then the family $X_{s,t}$ is called a {\em two-parameter flow of linear operators of $\R^d$ on $[a,b]$}.
 \end{definition}

 \begin{theorem}[Two-parameter flow generated by Young differential equations]\label{thm.flow}
%Suppose that the assumptions of Theorem \ref{theo2} are satisfied.
Assume that the conditions ${\textbf H}_1-{\textbf H}_3$ hold. 
 The family of Cauchy operators of \eqref{integral.eq.Stieltes} generates a two parameter flow of homeomorphisms of $\R^d$. Namely, for $-\infty< t_1\leq t_2<+\infty$ and $\omega\in \widetilde{C}^p(\R,\R^m)$  we define $X(t_1,t_2,\omega,\cdot)$ according to Definition \ref{dfn.Cauchy} and setting $X(t_2,t_1,\omega,\cdot) := X^{-1}(t_1,t_2,\omega,\cdot)$, then the family
 $X(t_1,t_2,\omega,\cdot)$, $t_1,t_2\in [0,T]$, is a two parameter flow of homeomorphisms of $\R^d$ on $[0,T]$. Furthermore, the flow is continuous. 
 \end{theorem}
 \begin{proof}
 Conditions (i)-(ii) of Definition \ref{dfn.2flow} follow from Theorem \ref{cauchy.operator}.
 
 Condition (iii) of Definition \ref{dfn.2flow} follows from the definition $X(t_2,t_1,\omega,\cdot) := X^{-1}(t_1,t_2,\omega,\cdot)$ for $t_1\leq t_2$.
 Actually, it is seen from the proof of Theorem \ref{cauchy.operator} that the inverse $X(t_2,t_1,\omega,\cdot)$ satisfies the backward equation \eqref{eqn.bw}.
 
 Condition (iv) of Definition \ref{dfn.2flow} follows from 
  the definition of the Cauchy operators and 
 Theorem  \ref{flow1}. 
 
 The continuity of the flow follows directly from Theorem \ref{contsolution}.
 \end{proof}

\section*{Acknowledgments}

This research is funded by Vietnam National Foundation for Science and Technology Development (NAFOSTED) under grant number  101.03-2014.42.

%%%%%%%%%%%%%%%%%%%%%%%%%%%%%%%%%%%%%%%%%%%%

\bibliography{literatur}
\bibliographystyle{abbrv}

\end{document}